\documentclass[a4paper, 11pt]{amsart}
\usepackage[utf8]{inputenc}
\usepackage[english]{babel}
\usepackage[T1]{fontenc}
\usepackage{amsmath, amssymb, amsfonts}
\usepackage[all]{xy}
\usepackage{enumerate}
\usepackage{graphicx}
\usepackage{geometry}
\usepackage{tikz}
\usetikzlibrary{arrows}
\usepackage{hyperref}
\usepackage{comment}

\title{The KK-theory of Fundamental C*-algebras}
\author{Pierre Fima and Emmanuel Germain}
\thanks{P.F. is partially supported by ANR grants OSQPI and NEUMANN.  E.G thanks CMI, Chennai for its support when part of this research was underway.}

\theoremstyle{plain}
\newtheorem{theorem}{Theorem}[section]
\newtheorem{proposition}[theorem]{Proposition}
\newtheorem{corollary}[theorem]{Corollary}
\newtheorem{lemma}[theorem]{Lemma}

\theoremstyle{definition}
\newtheorem{definition}[theorem]{Definition}
\newtheorem{example}[theorem]{Example}

\theoremstyle{remark}
\newtheorem{remark}[theorem]{Remark}

\DeclareMathOperator{\E}{E}

\DeclareMathOperator{\V}{V}

\newcommand{\EE}{\mathbb{E}}
\newcommand{\F}{\mathcal{F}}

\newcommand{\Gr}{\mathcal{G}}
\newcommand{\HH}{\mathcal{H}}

\newcommand{\LL}{\mathcal{L}}

\newcommand{\PP}{\mathcal{P}}

\newcommand{\T}{\mathcal{T}}
\newcommand{\Z}{\mathbb{Z}}

\newcommand{\rev}{\overline}

\newcommand{\ot}{\otimes}
\newcommand{\id}{\text{id}}

\begin{document}

\begin{abstract}
Given a graph of C*-algebras as defined in \cite{FF13}, we prove a long exact sequence in KK-theory similar to the one obtained by 
Pimsner in \cite{Pi86} for both the maximal and the vertex-reduced fundamental C*-algebras of the graph in 
the presence of possibly non GNS-faithful conditional expectations. We deduce from it the KK-equivalence between the full fundamental C*-algebra and the vertex-reduced fundamental C*-algebra even for non GNS-faithful conditional expectations. Our results unify, simplify and generalize all the previous results obtained before by Cuntz, Pimsner, Germain and Thomsen. It also generalizes the previous results of the authors on amalgamated free products.
\end{abstract}

\maketitle

\section{Introduction}

\noindent In 1986  the description of the $KK$-theory for some groups C*-algebras culminated in the computation by M. Pimsner of full and reduced crossed products by groups acting on trees \cite{Pi86} (or by the fundamental group of a graph of groups in Serre's terminology).  To go over the group situation has been difficult  and it relied heavily  on various generalizations of Voiculescu absorption theorem (see \cite{Th03} for the most general results in that direction).  Note also that G. Kasparov and G. Skandalis had another proof of Pimsner long exact sequence when studying KK-theory for buildings \cite{KS91}.

\vspace{0.2cm}

\noindent However the results we obtain here are based on a completely different point of view. Introduced in \cite{FF13}, the full or reduced fundamental C*-algebras of a graph of C*-algebras allows to treat on
equal footings  amalgamated  free products and HNN extensions (and in particular cross-product by the integers).  Let's describe its context. A graph of C*-algebras is a finite oriented graph with unital C*-algebras attached to its edges ($B_e$) and vertices ($A_v$) such that for any edge $e$  there are embeddings $r_e$ and $s_e$ of $B_e$  in $A_{r(e)}$ and  $A_{s(e)}$  with $r(e)$ the range of $e$ and $s(e)$ its source.
As for groups, the full fundamental C*-algebra of the graph is a quotient of the universal C*-algebra generated by the $A_v$ and unitaries $u_e$ such that $u_e^*s_e(b)u_e=r_e(b)$ for all $b\in B_e$.
In the presence of conditional expectations from $A_{s(e)}$ and $A_{r(e)}$ onto $B_e$, one can also construct various representations of the full fundamental C*-algebra on Hilbert modules over $A_v$ or $B_e$. 
It is the interplay with the representations that yields the tools we need to prove our results.

\vspace{0.2cm}

\noindent  In our previous paper \cite{FG15}, we first looked at one of the simplest graphs : one edge, two different endpoints. The full fundamental C*-algebra is then the full amalgamated free product.  When the conditional expectations are {\sl not} GNS-faithful, there are two possible reduced versions: the reduced free product of D. Voiculecscu, that we call the edge-reduced amalgamated free product and the vertex-reduced amalgamated free product we did construct in \cite{FG15}. We did show that the full amalgamated free product and the vertex-reduced amalgamated free product are always K-equivalent and we did exhibit a long exact sequence in KK-theory for both of them.

\vspace{0.2cm}

\noindent In this paper, we extend the results of \cite{FG15} to any fundamental C*-algebra of a finite graph of C*-algebras in the presence of conditional expectations, even non GNS-faithful ones.

\vspace{0.2cm}

\noindent Our first task is to introduce the good version of the reduced fundamental C*-algebra since there are several possible constructions of the reduced fundamental C*-algebra when the conditional expectations are not GNS-faithful and this fact was not clearly known to the authors in \cite{FF13} in which it was always assumed that the conditional expectations are GNS-faithful. The construction of the vertex-reduced fundamental C*-algebra is made in section $2$. We also describe in details its fundamental properties.

\vspace{0.2cm}

\noindent Our second task is to define the boundary maps in the long exact sequence. This will be done in the more possible natural way: by multiplication, in the Kasparov product sense, by some elements in $KK^1$ that we construct in a very natural and geometric way in section $3$. We also study the fundamental properties of these $KK^1$ elements which will be useful to prove the exactness of the sequence.

\vspace{0.2cm}

\noindent In section $4$ we prove our main result: the exactness of the sequence. This is done by induction, using the analogue of Serre's devissage process, the properties of our $KK^1$ elements and the initial cases: the amalgamated free product case which was done in \cite{FG15} and the HNN-extension case which can be deduce from the amalgamated free product case by a remark of Ueda \cite{Ue08}. Explicitly, if $C$ is any separable
C*-algebra, $P$ the full or reduced fundamental C*-algebra of the finite graph of C*-algebras $(\mathcal{G},A_p,B_e)$ then we have the two $6$-terms exact sequence, where $E^+$ is the set of positive edges and $V$ is the vertex set of the graph $\Gr$, 
$$\begin{array}{rcl}
   \bigoplus_{e\in E^+} KK^0(C,B_e)&\overset{\sum s_e^*-r_e^*}{\longrightarrow}\bigoplus_{p\in V} KK^0(C,A_p) \overset{}{\longrightarrow} &KK^0(C,P)\\
\uparrow&&\downarrow\\
KK^1(C,P)& \longleftarrow\bigoplus_{p\in V} KK^1(C,A_p) \overset{\sum s_e^*-r_e^*}{\longleftarrow}&  \bigoplus_{e\in E^+} KK^1(C,B_e) \\
\end{array}
$$
and
$$\begin{array}{rcl}
   \bigoplus_{e\in E^+} KK^0(B_e,C)&\overset{\sum {s_e}_*-{r_e}_*}{\longleftarrow}\bigoplus_{p\in V} KK^0(A_p,C) \overset{}{\longleftarrow} &KK^0(P,C)\\
\downarrow&&\uparrow\\
KK^1(P,C)& \longrightarrow\bigoplus_{p\in V} KK^1(A_p,C) \overset{\sum {s_e}_*-{r_e}_*}{\longrightarrow}&  \bigoplus_{e\in E^+} KK^1(B_e,C) \\
\end{array}
$$

\noindent In section $5$ we give some applications of our results. A direct Corollary of our results is that the full and vertex-reduced fundamental C*-algebra of a graph of C*-algebras are K-equivalent. This generalizes and simplifies the results of Pimsner about the KK-theory of crossed-products by groups acting on trees \cite{Pi86}. Also, our results imply that the fundamental quantum group of a graph of discrete quantum groups is K-amenable if and only if all the vertex quantum groups are K-amenable. This generalizes and simplifies the results of \cite{FF13}.
%%%%%%%%%%%%%%%%%%%%%%%%%%%%%%%%%%%%%%%%%%%%%%%%%%%%%%%%%%%
\section{Preliminaries}
%%%%%%%%%%%%%%%%%%%%%%%%%%%%%%%%%%%%%%%%%%%%%%%%%%%%%%%%%%%

\subsection{Notations and conventions} All C*-algebras and Hilbert modules are supposed to be separable. For a C*-algebra $A$ and a Hilbert $A$-module $H$ we denote by $\mathcal{L}_A(H)$ the C*-algebra a $A$-linear adjointable operators from $H$ to $H$ and by $\mathcal{K}_A(H)$ the sub-C*-algebra of $\mathcal{L}_A(H)$ consisting of $A$-compact operators. We write $L_A(a)\in\mathcal{L}_A(A)$ the left multiplication operator by $a\in A$. We use the term \textit{ucp} for unital completely positive. When $\varphi\,:\,A\rightarrow B$ is a ucp map the \textit{GNS construction of} is the unique, up to a canonical isomorphism, triple $(H,\pi,\xi)$ such that $H$ is a Hilbert $B$-module, $\pi\,:\,A\rightarrow\mathcal{L}_B(H)$ is a unital $*$-homomorphism and $\xi\in H$ is a vector such that $\pi(A)\xi\cdot B$ is dense in $H$ and $\langle\xi,\pi(a)\xi\cdot b\rangle=\varphi(a)b$. We refer the reader to \cite{Bl86} for basics notions about Hilbert C*-modules and KK-theory.

%%%%%%%%%%%%%%%%%%%%%%%%%%%%%%%%%%%%%%%%%%%%%%%%%%%%%%%%%%%%
\subsection{Some homotopies}
%%%%%%%%%%%%%%%%%%%%%%%%%%%%%%%%%%%%%%%%%%%%%%%%%%%%%%%%%%%%

\begin{lemma}\label{Lem-Homotopy}
Let $A$, $B$ be unital C*-algebras, $H$, $K$ Hilbert $B$-modules, $\pi\,:\,A\rightarrow\mathcal{L}_B(H)$, $\rho\,:\,A\rightarrow\mathcal{L}_B(K)$ unital $*$-homomorphisms and $F\in\mathcal{L}_B(H,K)$ a partial isometry such that $F\pi(a)-\rho(a)F\in\mathcal{K}_B(H,K)$ for all $a\in A$ and $F^*F-1\in\mathcal{K}_B(H)$. Then, $[(K,\rho,V)]=0\in KK^1(A,B)$, where $V=2FF^*-1$.
\end{lemma}

\begin{proof}
Let $\alpha:=[(K,\rho,V)]\in KK^1(A,B)$. For $t\in[0,1]$, define
$$U_t=\left(\begin{array}{cc}1-FF^*&0\\0&0\end{array}\right)+\cos(\pi t)\left(\begin{array}{cc}FF^*&0\\0&-1\end{array}\right)-\sin(\pi t)\left(\begin{array}{cc}0&F\\F^*&0\end{array}\right)\in\mathcal{L}_B(K\oplus H).$$
We have $U_0=\left(\begin{array}{cc}1&0\\0&-1\end{array}\right)$ and $U_1=-\left(\begin{array}{cc}V&0\\0&1\end{array}\right)$. Note that, for all $t\in [0,1]$, $U_t^*=U_t$ and,
\begin{eqnarray*}
U_t^2&=&\left(\begin{array}{cc}1-FF^*&0\\0&0\end{array}\right)+\cos(\pi t)^2\left(\begin{array}{cc}FF^*&0\\0&1\end{array}\right)+\sin(\pi t)^2\left(\begin{array}{cc}FF^*&0\\0&F^*F\end{array}\right)\\
&=&\left(\begin{array}{cc}1-FF^*&0\\0&0\end{array}\right)+\left(\begin{array}{cc}FF^*&0\\0&1\end{array}\right)+K_t=\left(\begin{array}{cc}1&0\\0&1\end{array}\right)+K_t,
\end{eqnarray*}
where $K_t=\sin(\pi t)^2\left(\begin{array}{cc}0&0\\0&F^*F-1\end{array}\right)\in\mathcal{K}_B(K\oplus H)$ for all $t\in[0,1]$, since $F^*F-1\in\mathcal{K}_B(H)$. Moreover, $U_t(\rho\oplus\pi)(a)-(\rho\oplus\pi)(a)U_t\in\mathcal{K}_B(K\oplus H)$ for all $a\in A$ since $F\pi(a)-\rho(a)F\in\mathcal{K}_B(H,K)$ for all $a\in A$. Consider the unique operators $U\in\mathcal{L}_{B\ot C([0,1])}(K\oplus H)\ot C([0,1]))$ and $K\in \mathcal{K}_{B\ot C([0,1])}(K\oplus H)\ot C([0,1]))$ such that the evaluation of $U$ at $t$ is $U_t$ and the evaluation of $K$ at $t$ is $K_t$ for all $t\in[0,1]$. In particular we have $U=U^*$ and $U^2=1+K$ and, since $U_t(\rho\oplus\pi)(a)-(\rho\oplus\pi)(a)U_t\in\mathcal{K}_B(K\oplus H)$ for all $a\in A$ and all $t\in[0,1]$, we have
$$U(\rho\oplus\pi)(a)\ot 1_{C([0,1])}-(\rho\oplus\pi)(a)\ot 1_{C([0,1])}U\in\mathcal{K}_{B\ot C([0,1])}((K\oplus H)\ot C([0,1]))\quad\text{for all}\,\,a\in A.$$
Hence we get an homotopy
$$\gamma=[((K\oplus H)\ot C([0,1]),(\rho\oplus\pi)\ot 1_{C([0,1])},U)]\in KK^1(A\ot C([0,1]),B\ot C([0,1]))$$
between $\gamma_0=[(K\oplus H,\rho\oplus\pi, U_0)]=[(K\oplus H,\rho\oplus\pi, \left(\begin{array}{cc}1&0\\0&-1\end{array}\right))]=0$ since the triple is degenerated and $\gamma_1=[(K\oplus H,\rho\oplus\pi, U_1)]=[(K\oplus H,\rho\oplus\pi, -\left(\begin{array}{cc}V&0\\0&1\end{array}\right))]$. Hence, $\gamma_1=x\oplus y$, where $x=[(K,\rho,-V)]=-\alpha$ and $y=[(H,\pi,-\id_H)]=0$, since the triple is degenerated.
\end{proof}

  %%%%%%%%%%%%%%%%%%%%%%%%%%%%%%%%%%%%%%%%%%
 \subsection{Fundamental C*-algebras}
   %%%%%%%%%%%%%%%%%%%%%%%%%%%%%%%%%%%%%%%%%%
\noindent In this section we recall the results and notations of \cite{FF13} and generalize the constructions to the case of non GNS-faithful conditional expectations.

\vspace{0.2cm}

\noindent If $\Gr$ is a graph in the sense of \cite[Def 2.1]{Se77}, its vertex set will be denoted $\V(\Gr)$ and its edge set will be denoted $\E(\Gr)$. We will always assume that $\Gr$ is at most countable. For $e\in\E(\Gr)$ we denote by $s(e)$ and $r(e)$ respectively the source and range of $e$ and by $\rev{e}$ the inverse edge of $e$. An \emph{orientation} of $\Gr$ is a partition $\E(\Gr) = \E^{+}(\Gr)\sqcup \E^{-}(\Gr)$ such that $e\in \E^{+}(\Gr)$ if and only if $\rev{e}\in \E^{-}(\Gr)$. We call $\Gr'\subset \Gr$ a \textit{connected subgraph} if $V(\Gr')\subset V(\Gr)$, $E(\Gr')\subset E(\Gr)$ such that $e\in E(\Gr')$ if and only if $\overline{e}\in E(\Gr')$ and the graph $\Gr'$ with the source map and inverse map given map the ones of $\Gr$ restricted to $E(\Gr')$ is a connected graph.

\vspace{0.2cm}

\noindent Let $(\Gr, (A_{q})_{q}, (B_{e})_{e})$ be a \textit{graph of unital C*-algebras}. This means that

\begin{itemize}
\item $\Gr$ is a connected graph.
\item For every $q\in\V(\Gr)$ and every $e\in\E(\Gr)$, $A_{q}$ and $B_{e}$ are unital C*-algebras.
\item For every $e\in \E(\Gr)$, $B_{\rev{e}} = B_{e}$.
\item For every $e\in\E(\Gr)$, $s_{e} : B_{e} \rightarrow A_{s(e)}$ is a unital faithful $*$-homomorphism.
\end{itemize}
For every $e\in \E(\Gr)$, we set $r_{e} = s_{\rev{e}} : B_{e} \rightarrow A_{r(e)}$, $B_{e}^{s} = s_{e}(B_{e})$ and $B_{e}^{r} = r_{e}(B_{e})$. Given a maximal subtree $\mathcal{T}\subset\Gr$ the \emph{maximal fundamental C*-algebra with respect to $\T$} is the universal C*-algebra generated by the C*-algebras $A_{q}$, $q\in\V(\Gr)$, and by unitaries $u_{e}$, $e\in\E(\Gr)$, such that
\begin{itemize}
\item For every $e\in\E(\Gr)$, $u_{\rev{e}} = u_{e}^{*}$.
\item For every $e\in\E(\Gr)$ and every $b\in B_{e}$, $u_{\rev{e}}s_{e}(b)u_{e}=r_{e}(b)$.
\item For every $e\in\E(\T)$, $u_{e}=1$.
\end{itemize}
This C*-algebra will be denoted by $P$ or $P_\Gr$. We will always view $A_p\subset P$ for all $p\in V(\Gr)$ since, as explain in the following remark, the canonical unital $*$-homomorphisms from $A_p$ to $P$ are all faithful.

\begin{remark}\label{rem:faithrep}
The C*-algebra $P$ is not zero and the canonical maps $\nu_p\,:\,A_p\rightarrow P$ are injective for all $p\in V(\Gr)$. This follows easily from the Voiculescu's absorption Theorem since we did assume all our C*-algebras separable and the graph $\Gr$ countable. Indeed, since $A_p$ is separable for all $p\in V(\Gr)$ and since $\Gr$ is at most countable we can representation faithfully all the $A_p$ on the same separable Hilbert space $H$. Write $\pi'_p\,:\, A_p\rightarrow \mathcal{L}(H)$ the faithful representation. Replacing $H$ by $H\ot H$ and $\pi'_p$ by $\pi'_p\ot\id$ if necessary, we may and will assume that $\pi'_p(A_p)\cap\mathcal{K}(H)=\{0\}$ for all $p\in V(\Gr)$. Write $C=\mathcal{L}(H)/\mathcal{K}(H)$ the Calkin algebra and $Q\,:\,\mathcal{L}(H)\rightarrow C$ the canonical surjection. Fix an orientation of $\Gr$. For $e\in E(\Gr)$ we have two faithful representations $\pi'_{s(e)}\circ s_e$ and $\pi'_{r(e)}\circ r_e$ of $B_e$ on $H$, both having trivial intersection with $K(H)$. By Voiculescu's absorption Theorem there exists, for all $e\in E^+(\Gr)$, a unitary $V_e\in C$ such that $Q\circ\pi'_{r(e)}(r_e(b))=V_e^*Q\circ\pi'_{s(e)}(s_e(b))V_e$ for all $b\in B_e$ and all $e\in E^+(\Gr)$. For $e\in E^-(\Gr)$ define $V_e:=(V_{\overline{e}})^*$ so that the relations $(V_e)^*=V_{\overline{e}}$ and $Q\circ\pi'_{r(e)}(r_e(b))=V_e^*Q\circ\pi'_{s(e)}(s_e(b))V_e$ holds for all $b\in B_e$ and all $e\in E(\Gr)$. When $\omega=(e_1,\dots e_n)$ is a path in $\Gr$, we denote by $V_\omega$ the unitary $V_\omega:=V_{e_1}\dots V_{e_n}\in C$ (if $\omega$ is the empty path we put $V_\omega=1$). Fix a maximal subtree $\mathcal{T}\subset\mathcal{G}$. For $p,q\in V(\Gr)$ let $g_{pq}$ be the unique geodesic path in $\mathcal{T}$ from $p$ to $q$ (if $p=q$ then $g_{pq}$ is the empty path by convention). Fix $p_0\in V(\mathcal{G})$ and, for $e\in E(\mathcal{G})$, define $U_e:=(V_{g_{s(e)}p_0})^*V_{(e,g_{r(e)p_0})}$ so that the relations $U_{\overline{e}}=U_e^{*}$ holds for any $e\in E(\Gr)$ and $U_e=1$ for any $e\in E(\mathcal{T})$. Finally, for $p\in V(\Gr)$, define the faithful (since $\pi_p'(A_p)\cap\mathcal{K}(H)=\{0\}$) unital $*$-homomorphism $\pi_p\,:\,A_p\rightarrow C$ by $\pi_p:=(V_{g_{p_0p}})^* Q\circ\pi_p'(\cdot)V_{g_{p_0p}}$. Then, it is easy to check that the relation $\pi_{r(e)}(r_e(b))=U_e^*\pi_{s(e)}(s_e(b))U_e$ holds for all $b\in B_e$ and all $e\in E(\Gr)$. Hence, $P$ is not zero and we have a unique unital $*$-homomorphism $\pi\,:\, P\rightarrow C$ such that $\pi(u_e)=U_e$ and $\pi\circ\nu_p=\pi_p$ for all $p\in V(\Gr)$. In particular, the canonical map $\nu_p$ from $A_p$ to $P$ is faithful since $\pi_p$ is faithful. Note that, when the C*-algebras $A_p$ are not supposed to be separable and/or the graph $\Gr$ is not countable anymore the result is still true by considering the universal representation, as in the proof of \cite[Theorem 4.2]{P1} (which was inspired by \cite{Bl78}).
\end{remark}

\begin{remark}\label{rem:pathmax}
Let $\mathcal{A}\subset P$ be the $*$-algebra generated by the $A_q$, for $q\in V(\Gr)$, and the unitaries $u_e$, for $e\in E(\Gr)$. Then $\mathcal{A}$ is a dense unital $*$-algebra of $P$. Moreover, since the graph $\Gr$ is supposed to be connected, for any fixed $p\in \V(\Gr)$, $\mathcal{A}$ is the linear span of $A_{p}$ and elements of the form $a_{0}u_{e_{1}} \dots u_{e_{n}}a_{n}$ where $(e_{1}, \dots, e_{n})$ is a path in $\Gr$ from $p$ to $p$, $a_{0}\in A_{p}$ and $a_{i}\in A_{r(e_{i})}$ for $1\leqslant i\leqslant n$.\end{remark}

\noindent We now recall the construction of the reduced fundamental C*-algebra, when there is a family of conditional expectations $E_e^s\,:\,A_{s(e)}\rightarrow B_e^s$, for $e\in E(\Gr)$. Set $E_{e}^{r} = E_{\rev{e}}^{s} : A_{r(e)} \rightarrow B_{e}^{r}$ and note that, in contrast with \cite{FF13}, we do not assume the conditional expectations $E_e^s$ to be GNS-faithful. However, as it was already mentioned in \cite{FF13}, all the constructions can be easily carried out without the non-degeneracy assumption. Let us recall these constructions now. We shall omit the proofs which are exactly the same as the GNS-faithful case and concentrate only on the differences with the GNS-faithful case.

\vspace{0.2cm}

\noindent For every $e\in\E(\Gr)$ let $(K_{e}^{s},\rho_{e}^{s},\eta_{e}^{s})$ be the GNS construction of the ucp map $s_{e}^{-1}\circ E_{e}^{s}\,:\,A_{s(e)}\rightarrow B_e$. This means that $K_{e}^{s}$ is a right Hilbert $B_{e}$-module, $\rho_e^s\,:\,A_{s(e)}\rightarrow\mathcal{L}_{B_e}(K_e^s)$ and $\eta_e^s\in K_e^s$ are such that $K_e^s=\overline{\rho_e^s(A_{s(e)})\eta_e^s\cdot B_e}$ and $\langle\eta_e^s,\rho_e^s(a)\eta_e^s\cdot b\rangle=s_e^{-1}\circ E_e^s(a)b$. In particular, we have the formula $\rho_e^s(a)\eta_e^s\cdot b=\rho_e^s(as_e(b))\eta_e^s$. Let us notice that the submodule $\eta_{e}^{s}.B_{e}$ of $K_{e}^{s}$ is orthogonally complemented. In fact, its orthogonal complement $(K_{e}^{s})^{\circ}$ is the closure of the set $\{\rho_e^s(a)\eta_e^s\,:\,a\in A_{s(e)}, E_{e}^{s}(a)=0\}$ which is easily seen to be a Hilbert $B_e$-submodule of $K_e^s$. Similarly, the orthogonal complement of $\eta_{e}^{r}.B_{e}$ in $K_{e}^{r}$ will be denoted $(K_{e}^{r})^{\circ}$. Note that $\rho_e^s(B_e^s)(K_{e}^{s})^{\circ}\subset (K_{e}^{s})^{\circ}$.

\vspace{0.2cm}

\noindent Let $n\geqslant 1$ and $w = (e_{1}, \dots, e_{n})$ a path in $\Gr$. We define Hilbert C*-modules $K_{i}$ for $0\leqslant i\leqslant n$ by
\begin{itemize}
\item $K_{0} = K_{e_{1}}^{s}$
\item If $e_{i+1}\neq \rev{e}_{i}$, then $K_{i} = K_{e_{i+1}}^{s}$
\item If $e_{i+1}= \rev{e}_{i}$, then $K_{i} = (K_{e_{i+1}}^{s})^{\circ}$
\item $K_{n} = A_{r(e_{n})}$
\end{itemize}
For $0\leqslant i\leqslant n-1$, $K_{i}$ is a right Hilbert $B_{e_{i+1}}$-module and $K_{n}$ will be seen as a right Hilbert $A_{r(e_{n})}$-module. We define, for $1\leqslant i \leqslant n-1$, the unital $*$-homomorphism
$$\rho_{i} = \rho_{e_{i+1}}^{s}\circ r_{e_{i}} : B_{e_{i}} \rightarrow \LL_{B_{e_{i+1}}}(K_{i}),$$
and, $\rho_n=L_{A_{r(e_n)}}\circ r_{e_n}\,:\,B_{e_{n}} \rightarrow \LL_{A_{r(e_{n})}}(K_{n})$. We can now define the right Hilbert $A_{r(e_n)}$-module
\begin{equation*}
H_{w}=K_{0}\underset{\rho_1}{\otimes} \dots \underset{\rho_n}{\otimes} K_{n}
\end{equation*}
endowed with the left action of $A_{s(e_{1})}$ given by the unital $*$-homomorphism defined by
$$\lambda_w=\rho_{e_1}^s\ot\id\,:\, A_{s(e_1)}\rightarrow\LL_{A_{r(e_n)}}(H_w).$$

\noindent For any two vertices $p, q\in \V(\Gr)$, we define the Hilbert $A_p$-module $H_{q,p}=\bigoplus_{w}H_{w}$, where the sum runs over all paths $w$ in $\Gr$ from $q$ to $p$. By convention, when $q=p$, the sum also runs over the empty path, where $H_{\emptyset} = A_{p}$ with its canonical Hilbert $A_{p}$-module structure. We equip this Hilbert C*-module with the left action of $A_{q}$ which is given by  $\lambda_{q,p}\,:\, A_{q}\rightarrow\LL_{A_p}(H_{q,p})$ defined by $\lambda_{q,p}=\bigoplus_{w}\lambda_{w}$, where, when $q=p$ and $w=\emptyset$ is the empty path, $\lambda_{\emptyset}:=L_{A_p}$.

\vspace{0.2cm}

\noindent For every $e\in\E(\Gr)$ and $p\in\V(\Gr)$, we define an operator $u_{e}^{p} : H_{r(e),p} \rightarrow H_{s(e),p}$ in the following way. Let $w$ be a path in $\Gr$ from $r(e)$ to $p$ and let $\xi\in \HH_w$.
\begin{itemize}
\item If $p = r(e)$ and $w$ is the empty path, then $u_{e}^{p}(\xi) = \eta_{e}^{s}\otimes\xi \in H_{(e)}$.
\item If $n = 1$, $w = (e_{1})$, $\xi = \rho_{e_1}^s(a)\eta_{e_1}^s\otimes\xi'$ with $a\in A_{s(e_{1})}$ and $\xi'\in A_{p}$, then
\begin{itemize}
\item If $e_{1} \neq \rev{e}$, $u_{e}^{p}(\xi) = \eta_{e}^{s}\otimes\xi \in H_{(e, e_{1})}$.
\item If $e_{1} = \rev{e}$, $u_{e}^{p}(\xi) =
\left\{\begin{array}{cccc}
\eta_{e}^{s}\otimes\xi & \in H_{(e, e_{1})} & \text{if} & E_{e_1}^s(a)=0, \\
r_{e_{1}}\circ s_{e_{1}}^{-1}(a)\xi' & \in A_{p} & \text{if} & a\in B^{s}_{e_{1}}.
\end{array}\right.$
\end{itemize}
\item If $n\geqslant 2$, $w = (e_{1}, \dots, e_{n})$, $\xi = \rho_{e_1}^s(a)\eta_{e_1}^s\otimes\xi'$ with $a\in A_{s(e_{1})}$ and $\xi' \in K_{1} \underset{\rho_{2}}{\otimes} \dots \underset{\rho_{n}}{\otimes} K_{n}$, then
\begin{itemize}
\item If $e_{1} \neq \rev{e}$, $u_{e}^{p}(\xi) = \eta_{e}^{s}\otimes\xi \in H_{(e, e_{1}, \dots, e_{n})}$.
\item If $e_{1} = \rev{e}$, $u_{e}^{p}(\xi) =
\left\{\begin{array}{cccc}
\eta_{e}^{s}\otimes\xi & \in H_{(e, e_{1}, \dots, e_{n})} & \text{if} & E_{e_1}^s(a)=0, \\
(\rho_1(s_{e_{1}}^{-1}(a))\ot\id)\xi' & \in H_{(e_{2}, \dots, e_{n})} & \text{if} & a \in B_{e_{1}}^{s}.\end{array}\right.$
\end{itemize}
\end{itemize}
One easily checks that the operators $u_{e}^{p}$ commute with the right actions of $A_{p}$ on $H_{r(e), p}$ and $H_{s(e), p}$ and extend to unitary operators (still denoted $u_{e}^{p}$) in $\mathcal{L}_{A_p}(H_{r(e),p},H_{s(e),p})$ satisfying $(u_{e}^{p})^{*} = u_{\rev{e}}^{p}$. Moreover, for every $e\in \E(\Gr)$ and every $b\in B_{e}$, the definition implies that
\begin{equation*}
u_{\rev{e}}^{p}\lambda_{s(e),p}(s_{e}(b))u_{e}^{p} = \lambda_{r(e),p}(r_{e}(b))\in\mathcal{L}_{A_p}(H_{r(e),p}).
\end{equation*}
Let $w=(e_1,\ldots, e_n)$ be a path in $\Gr$ and let $p\in\V(\Gr)$, we set $u^{p}_{w} = u_{e_{1}}^p \dots u_{e_{n}}^{p} \in \LL_{A_{p}}(H_{r(e_{n}),p},H_{s(e_{1}),p}).$

\vspace{0.2cm}

\noindent The $p$-\emph{reduced fundamental C*-algebra} is the C*-algebra
$$
P_p =  \left\langle (u^{p}_{z})^{*}\lambda_{q,p}(A_{q})u^{p}_{w} \vert q \in \V(\Gr),w,z \text{ paths from $q$ to $p$ }\right\rangle
\subset \LL_{A_{p}}(H_{p,p}).
$$

\noindent We sometimes write $P_p^\Gr=P_p$. Let us now explain how one can canonically view $P_p$ as a quotient of $P$. Let $\mathcal{T}$ be a maximal subtree in $\Gr$. Given a vertex $q\in \V(\Gr)$, we denote by $g_{qp}$ the unique geodesic path in $\mathcal{T}$ from $q$ to $p$. For every $e\in \E(\Gr)$, we define a unitary operator $w^{p}_{e}=(u^{p}_{g_{s(e)p}})^{*}u^{p}_{(e, g_{r(e)p})}\in P_p$.

\noindent For every $q\in\V(\Gr)$, we define a unital faithful $*$-homomorphism $\pi_{q, p}: A_{q} \rightarrow P_{p}$ by
\begin{equation*}
\pi_{q, p}(a) = (u^{p}_{g_{qp}})^{*}\lambda_{q,p}(a) u^{p}_{g_{qp}}\quad\text{for all  }a\in A_q.
\end{equation*}
It is not difficult to check that the following relations hold:
\begin{itemize}
\item $w^{p}_{\rev{e}} = (w^{p}_{e})^{*}$ for every $e\in\E(\Gr)$,
\item $w^{p}_{e} = 1$ for every $e\in \E(\T)$,
\item $w^p_{\rev{e}}\pi_{s(e), p}(s_e(b))w^p_{e} = \pi_{r(e), p}(r_e(b))$ for every $e\in \E(\Gr)$, $b\in B_{e}$.
\end{itemize}

\noindent Thus, we can apply the universal property of the maximal fundamental C*-algebra $P$ to get a unique surjective $*$-homomorphism $\lambda_{p}\,:\, P \rightarrow P_{p}$ such that $\lambda_p(u_e)=w_e^p$ for all $e\in\E(\Gr)$ and $\lambda_p(a)=\pi_{q,p}(a)$ for all $a\in A_q$ and all $q\in V(\Gr)$. We sometimes write $\lambda_p^\Gr=\lambda_p$.

\vspace{0.2cm}

\noindent Let $p_0,p,q\in V(\Gr)$ and $a=\lambda_{p_0,p}(a_{0})u^{p}_{e_{1}}\lambda_{s(e_2),p}(a_1)u^p_{e_2} \dots u^{p}_{e_{n}}\lambda_{q,p}(a_{n})\in\mathcal{L}_{A_p}(H_{q,p},H_{p_0,p})$, where $w = (e_{1}, \dots,  e_{n})$ is a (non-empty) path in $\Gr$ from $p_{0}$ to $q$, $a_{0}\in A_{p_{0}}$ and, for $1\leqslant i\leqslant n$, $a_{i}\in A_{r(e_{i})}$. The operator $a$ is said to be \emph{reduced} (from $p_0$ to $q$) if for all $1\leqslant i\leqslant n-1$ such that $e_{i+1} = \rev{e}_{i}$, we have $\EE_{e_{i+1}}^{s}(a_{i}) = 0$.

\vspace{0.2cm}

\noindent Let $w=(e_{1}, \dots,  e_{n})$ be a path from $p$ to $p$. It is easy to check that any reduced operator of the form $a = \lambda_{p_0,p}(a_{0})u^{p}_{e_{1}} \dots u^{p}_{e_{n}}\lambda_{q,p}(a_{n})$ is in $P_p$ and that the linear span $\mathcal{A}_p$ of $A_{p}$ and the reduced operators from $p$ to $p$ is a dense $*$-subalgebra of $P_p$.

\begin{remark}\label{rem:redmax}
The notion of reduced operator also makes sense in the maximal fundamental C*-algebra (if we assume the existence of conditional expectations) and, for any fixed $p\in V(\Gr)$, the linear span of $A_{p}$ and the reduced operators from $p$ to $p$ is the $*$-algebra $\mathcal{A}$ introduced in Remark \ref{rem:pathmax}, which is dense in the maximal fundamental C*-algebra. Observe that, by definition, the morphism $\lambda_p\,:\, P\rightarrow P_p$ is the unique unital $*$-homomorphism which is formally equal to the identity on the reduced operators from $p$ to $p$. More precisely, one has, for any reduced operator $a=a_0u_{e_1}\dots u_{e_n}a_n\in P$ from $p$ to $p$, $\lambda_p(a)=\lambda_{p,p}(a_0)u_{e_1}^p\dots u_{e_n}^p\lambda_{p,p}(a_n)$.
\end{remark}

\noindent We will need the following purely combinatorial lemma which gives an explicit decomposition of the product of two reduced operators in $P$ from $p$ to $p$ as a sum of reduced operators from $p$ to $p$ plus an element in $A_{p}$. For $e\in E(\Gr)$ and $x\in A_{r(e)}$ we write $\PP_e^r(x):=x-E_e^r(x)$.

\begin{lemma}\cite[Lemma 3.17]{FF13}\label{LemProduct}
Let $w=(e_{n}, \dots, e_{1})$ and $\mu = (f_{1}, \dots, f_{m})$ be paths from $p$ to $p$. Set $n_{0} = \max\{1\leqslant k\leqslant\min(n, m)\vert e_{i} = \rev{f}_{i}, \forall i\leqslant k\}$. If the above set is empty, set $n_{0} = 0$. Let $a = a_{n} u_{e_{n}} \dots u_{e_{1}}a_{0}\in P$ and $b=b_0u_{f_1}\dots u_{f_m} b_m\in P$ be reduced operators. Set $x_{0} = a_{0}b_{0}$ and, for $1\leqslant k\leqslant n_{0}$, $x_{k} = a_{k}(s_{e_{k}}\circ r_{e_{k}}^{-1}\circ E_{e_{k}}^r(x_{k-1}))b_{k}$ and $y_{k} =\PP_{e_k}^{r}(x_{k-1})$. The following holds :
\begin{enumerate}

\item If $n_{0} = 0$, then $ab = a_{n}u_{e_{n}} \dots u_{e_{1}}x_{0}u_{f_{1}} \dots u_{f_{m}}b_{m}$.

\item If $n_0=n=m$, then $ab=\sum_{k=1}^na_{n}u_{e_{n}} \dots u_{e_{k}}y_{k}u_{f_{k}} \dots u_{f_{n}}b_{n}+x_n.$

\item If $n_{0} = n<m$, then $ab = \sum_{k = 1}^{n}a_{n}u_{e_{n}} \dots u_{e_{k}}y_{k}u_{f_{k}} \dots u_{f_{m}}b_{m} + x_{n}u_{f_{n+1}} \dots u_{f_{m}}b_{m}$.

\item If $n_0=m<n$, then $ab = \sum_{k = 1}^{m}a_{n}u_{e_{n}} \dots u_{e_{k}}y_{k}u_{f_{k}} \dots u_{f_{m}}b_{m} + a_{n} u_{e_{n}} \dots u_{e_{m+1}} x_{m}$.

\item If $1\leqslant n_0<\min\{n,m\}$, then
$$ab = \sum_{k = 1}^{n}a_{n}u_{e_{n}} \dots u_{e_{k}}y_{k}u_{f_{k}} \dots u_{f_{m}}b_{m} + a_{n} u_{e_{n}}\dots u_{e_{n_{0}+1}}x_{n_{0}}u_{f_{n_{0}+1}} \dots u_{f_{m}}b_{m}.$$
\end{enumerate}
\end{lemma}

\noindent Note that the preceding Lemma also holds in $P_p$, for all $p\in V(\Gr)$, simply by applying the unital $*$-homomorphism $\lambda_p$ which is formally the identity on the reduced operators from $p$ to $p$, as explained in Remark \ref{rem:redmax}.

\vspace{0.2cm}

\noindent In the following Proposition we completely characterize the $p$-reduced fundamental C*-algebra: it is the unique quotient of $P$ for which there exists a GNS-faithful ucp map $P_p\rightarrow A_p$ which is zero on the reduced operators and "the identity on $A_p$". The proof of this result is contained in \cite{FF13} in the GNS-faithful case but it is not explicitly stated. Since the proof is the same as the one of \cite[Proposition 2.4]{FG15} and all the necessary arguments are contained in \cite{FF13}, we will only sketch the proof of the next Proposition.

\begin{proposition}\label{Prop-pvertex}
For all $p\in V(\Gr)$ the following holds.
\begin{enumerate}
\item The morphism $\lambda_p$ is faithful on $A_p$.
\item There exists a unique ucp map $\EE_p\,:\, P_p\rightarrow A_p$ such that $\EE_p\circ\lambda_p(a)=a$ for all $a\in A_p$ and 
$$\EE_p(\lambda_p(a_0u_{e_1}\dots u_{e_n} a_n))=0\text{ for all }a=a_0u_{e_1}\dots u_{e_n} a_n\in P\text{ a reduced operator from }p\text{ to }p.$$
Moreover, $\EE_p$ is GNS-faithful.
\item For any unital C*-algebra with a surjective unital $*$-homomorphism $\pi\,:\,P\rightarrow C$ and a GNS-faithful ucp map $E\,:\,C\rightarrow A_p$ such that $E\circ\lambda(a)=a$ for all $a\in A_p$ and
$$E(\pi(a_0u_{e_1}\dots u_{e_n} a_n))=0\text{ for all }a=a_0u_{e_1}\dots u_{e_n} a_n\in P\text{ a reduced operator from }p\text{ to }p$$
there exists a unique unital $*$-isomorphism $\nu\,:\, P_p\rightarrow C$ such that $\nu\circ\lambda_p=\pi$. Moreover, $\nu$ satisfies $E\circ\nu=\EE_p$.
\end{enumerate}
\end{proposition}

\begin{proof}
Assertion $(1)$ follows from assertion $(2)$, since $\EE_p\circ\lambda_p(a)=a$ for all $a\in A_p$. Let us sketch the proof of assertion $(2)$. Define $\xi_p=1_{A_p}\in A_p\subset H_{p,p}$ and $\EE_p(x)=\langle \xi_p,x\xi_p\rangle$ for all $x\in P_p$. Then $\EE_p\,:\, P_p\rightarrow A_p$ is a ucp map and, for all $a\in A_p$, $\EE_p(\lambda_p(a))=\langle 1_{A_p},L_{A_p}(a)1_{A_p}\rangle=a$. Repeating the proof of \cite[Proposition 3.18]{FF13}, we see that $\overline{P_p\xi_p\cdot A_p}=H_{p,p}$ and, for any reduced operator $a\in A_p$, one has $\langle \xi_p,a\xi_p\rangle=0$. It follows that the triple $(H_{p,p},\id,\xi_p)$ is a GNS-construction of $\EE_p$ (in particular $\EE_p$ is GNS-faithful) and $\EE_p(\lambda_p(x))=0$ for any reduced operator $x\in P$ from $p$ to $p$, since the map $\lambda_p$ sends reduced operators in $P$ from $p$ to $p$ to reduced operators in $P_p$.

\vspace{0.2cm}

\noindent The proof of $(3)$ is a routine. Since $E$ is GNS-faithful on $C$ we may and will assume that $C\subset\mathcal{L}_{A_p}(K)$, where $(K,\id,\eta)$ is a GNS-construction of $E$. By the properties of $E$ and $\EE_p$, the operator $U\,:\, H_{p,p}\rightarrow K$ defined by $U(\lambda_p(x)\xi_p)=\pi(x)\eta$ for all $x\in P$ reduced operator from $p$ to $p$ or $x\in A_p\subset P$ extends to a unitary operator $U\in\mathcal{L}_{A_p}(H_{p,p},K)$. By the definition of $U$, the map $\nu(x)=UxU^*$, for $x\in P_p$, does the job. The uniqueness is obvious.
\end{proof}

\noindent\textit{Notation.} We sometimes write $\EE_{p}^\Gr=\EE_{p}$.

\vspace{0.2cm}

\noindent For a connected subgraph $\mathcal{G}'\subset\mathcal{G}$ with a maximal subtree $\T'\subset\Gr'$ such that $\T'\subset\T$ we denote by $P_{\Gr'}$ the maximal fundamental C*-algebra of our graph of C*-algebras restricted to $\Gr'$ with respect to the maximal subtree $\mathcal{T}'$. By the universal property there exists a unique unital $*$-homomorphism $\pi_{\Gr'}\,:\,P_{\Gr'}\rightarrow P$ such that $\lambda_{\Gr'}(a)=a$ for all $a\in A_p$, $p\in V(\Gr')$ and $\pi_{\Gr'}(u_e)=u_e$ for all $e\in E(\Gr')$. The following Corollary says that we have a canonical identification of $P_p^{\Gr'}$ with the sub-C*-algebra of $P_p$ generated by $A_p$ and the reduced operators from $p$ to $p$ with associated path in $\Gr'$.

\begin{proposition}\label{Cor-Univ-pred}
For all $p\in V(\Gr')$, there exists unique faithful $*$-homomorphism $\pi_p^{\Gr'}\,:\, P_p^{\Gr'}\rightarrow P_p$ such that $\pi_p^{\Gr'}\circ\lambda_p^{\Gr'}=\lambda_p\circ\pi_{\Gr'}$. The morphism $\pi_{p}^{\Gr'}$ satisfies $\EE_p\circ\pi_p^{\Gr'}=\EE_p^{\Gr'}$. Moreover, there exists a unique ucp map $\EE_p^{\Gr'}\,:\,P_{p}\rightarrow P_{p}^{\Gr'}$ such that $\EE_p^{\Gr'}\circ\pi_p^{\Gr'}=\id$ and $\EE_p^{\Gr'}(\lambda_p(a))=0$ for all $a\in P$ a reduced operator from $p$ to $p$ with associated path containing at least one vertex which is not in $\Gr'$.
\end{proposition}

\begin{proof}
The uniqueness of $\pi_p^{\Gr'}$ being obvious, let us show the existence. Define $P'=\pi_p^{\Gr'}\circ\lambda_p^{\Gr'}(P_{\Gr'})\subset P_p$ and let $E\,:\, P'\rightarrow A_p$ be the ucp map defined by $E=\EE_p\vert_{P'}$. By the universal property of Proposition \ref{Prop-pvertex}, assertion $3$, it suffices to check that $E$ is GNS-faithul. Let $x\in P'$ such that $E(y^*x^*xy)=\EE_p(y^*x^*xy)=0$ for all $y\in P'$. In particular $\EE_p(x^*x)=0$ and we may and will assume that $x^*x$ is the image under $\lambda_p$ of a sum of reduced operators from $p$ to $p$ with associated vertices in $\Gr'$. Let us show that $x=0$. Since $\EE_p$ is GNS-faitful and since $P'$ contains the image under $\lambda_p$ of $A_p$ and of the reduced operators from $p$ to $p$ in $P$ whose associated path from is in $\Gr'$, it suffices to show that $\EE_p(y^*x^*xy)=0$ for all $y=\lambda_p(a)$, where $a\in P$ is a reduced operator from $p$ to $p$ whose associated path contains at least one vertex which is not in $\Gr'$. It follows easily from Lemma \ref{LemProduct} since this Lemma implies that, for all $z\in P$ a reduced operator from $p$ to $p$ with all edges in $\Gr'$ or $z\in A_p$ and for all $a\in P$ a reduced operator from $p$ to $p$ with at least one vertex which is not in $\Gr'$, the product $a^*za$ is equal to a sum of reduced operators from $p$ to $p$ with at least one vertex which is not in $\Gr'$. In particular, $\EE_p(\lambda_p(a^*za))=0$ for all such $a$ and $z$. Hence, $\EE_p(yx^*xy)=0$ for all $y\in P_p$. By construction, $\pi_{p}^{\Gr'}$ satisfies $\EE_p\circ\pi_p^{\Gr'}=\EE_p^{\Gr'}$. Let us now construct the ucp map $\EE_p^{\Gr'}$ (the uniqueness is obvious).

\vspace{0.2cm}

\noindent Let $H_{p,p}'=\underset{\omega\text{ a path in }\Gr'\text{ from }p\text{ to }p}{\bigoplus}H_\omega\subset H_{p,p}$. By convention the sum also contains the empty path for which $H_\emptyset=A_{p}$. Observe that $H'_{p,p}$ is a complemented Hilbert sub-$A_{p}$-module of $H_{p,p}$. Let $Q\in\mathcal{L}_{A_{p}}(H_{p,p})$ be the orthogonal projection onto $H'_{p,p}$ and define the ucp map $\EE_p^{\Gr'}\,:\,P_{p}\rightarrow \mathcal{L}_{A_{p}}(H'_{p,p})$ by $\EE_p^{\Gr'}(x)=QxQ$.

\vspace{0.2cm}

\noindent Since $xH'_{p,p}\subset H'_{p,p}$ for all $x\in P_{p}^{\Gr'}$, the projection $Q$ commutes with every $x\in P_{p}^{\Gr'}$. Hence, after the identification $P_{p}^{\Gr'}\subset P_{p}$, we have $\EE_p^{\Gr'}(x)=x$ for all $x\in P_{p}^{\Gr'}$.

\vspace{0.2cm}

\noindent Let $a=a_0u_{e_1}\dots u_{e_n}a_n\in P$ be a reduced operator with $\omega=(e_1,\dots e_n)$ a path in $\Gr$ from $p$ to $p$ such that $e_k\notin E(\Gr')$ for some $1\leq k\leq n$. Observe that, by Lemma \ref{LemProduct}, for all $b\in P$ a reduced operator from $p$ to $p$ with associated path in $\Gr'$ or for $b\in A_{p}$ the product $ab$ is a sum of reduced operators from $p$ to $p$ whose associated path has at least one edge from $\Gr'$. Hence, $\lambda_p(ab)\xi_p\in H_{p,p}\ominus H'_{p,p}$ (where $\xi_p=1_{A_p}\in H_{p,p}$). It follows now easily from this observation that $Q\lambda_p(a)Q\lambda_p(b)\xi_p=0$ for all $b\in P$ a reduced operator from $p$ to $p$ or $b\in A_{p}$. Hence, $Q\lambda_p(a)Q=0$ and this concludes the proof.
\end{proof}
\noindent The following definition is not contained in \cite{FF13}. It is the correct version of the reduced fundamental C*-algebra in the case of non GNS-faithful conditional expectations in order to obtain the K-equivalence with the full fundamental C*-algebra. It is the main contribution of this preliminary section to the general theory of fundamental C*-algebras developed in \cite{FF13}.
\begin{definition} The \textit{vertex-reduced fundamental C*-algebra} $P_{{\rm vert}}$ is the C*-algebra obtained by separation completion of $P$ for the C*-semi-norm $\Vert x\Vert_v=\text{Sup}\{\Vert\lambda_p(x)\Vert\,:\,p\in V(\Gr)\}$ on $P$.
\end{definition}
\noindent We sometimes write $P_{\rm vert}^\Gr=P_{{\rm vert}}$. We will denote by $\lambda\,:\,P\rightarrow P_{{\rm vert}}$ (or $\lambda_\Gr$) the canonical surjection. Note that, by construction of $P_{{\rm vert}}$, for all $p\in V(\Gr)$, there exists a unique unital (surjective) $*$-homomorphism $\lambda_{v,p}\,:\,P_{{\rm vert}}\rightarrow P_p$ such that $\lambda_{v,p}\circ\lambda=\lambda_p$. We sometimes write $\lambda_{v,p}^\Gr=\lambda_{v,p}$. We describe the fundamental properties of $P_{{\rm vert}}$ in the following Proposition. We call a family of ucp maps $\{\varphi_i\}_{i\in I}$, $\varphi_i\,:\,A\rightarrow B_i$ GNS-faithful if $\cap_{i\in I}{\rm Ker}(\pi_i)=\{0\}$, where $(H_i,\pi_i,\xi_i)$ is a GNS-construction for $\varphi_i$.

\begin{proposition}\label{Prop-vertex}
The following holds.
\begin{enumerate}
\item The morphism $\lambda$ is faithful on $A_p$ for all $p\in V(\Gr)$.
\item For all $p\in V(\Gr)$, there exists a unique ucp map $\EE_{A_p}\,:\, P_{\rm vert}\rightarrow A_p$ such that $\EE_{A_p}\circ\lambda(a)=a$ for all $a\in A_p$ and all $p\in V(\Gr)$ and, 
$$\EE_{A_p}(\lambda_v(a_0u_{e_1}\dots u_{e_n} a_n))=0\text{ for all }a=a_0u_{e_1}\dots u_{e_n} a_n\in P\text{ a reduced operator from }p\text{ to }p.$$
Moreover, the family $\{\EE_{A_p}\,:\,p\in V(\Gr)\}$ is GNS-faithful.
\item Suppose that $C$ is a unital C*-algebra with a surjective unital $*$-homomorphism $\pi\,:\,P\rightarrow C$ and with ucp maps $E_{A_p}\,:\,C\rightarrow A_p$, for $p\in V(\Gr)$, such that $E_{A_p}\circ\pi(a)=a$ for all $a\in A_p$, all $p\in V(\Gr)$ and,
$$E_{A_p}(\pi(a_0u_{e_1}\dots u_{e_n} a_n))=0\text{ for all }a=a_0u_{e_1}\dots u_{e_n} a_n\in P\text{ a reduced operator from }p\text{ to }p$$
and the family $\{E_{A_p}\,:\,p\in V(\Gr)\}$ is GNS-faithful. Then, there exists a unique unital $*$-isomorphism $\nu\,:\, P_{{\rm vert}}\rightarrow C$ such that $\nu\circ\lambda=\pi$. Moreover, $\nu$ satisfies $E\circ\nu=\EE_p$ for all $p\in V(\Gr)$.
\end{enumerate}
\end{proposition}

\begin{proof}
$(1)$. It follows from $(2)$ since $\EE_{A_p}\circ\lambda(a)=a$ for all $a\in A_p$ and all $p\in V(\Gr)$.

\vspace{0.2cm}

\noindent$(2)$. By Proposition \ref{Prop-pvertex}, the maps $\EE_{A_p}=\EE_p\circ\lambda_{v,p}$ satisfy the desired properties and it suffices to check that the family $\{\EE_{A_p}\,:\,p\in V(\Gr)\}$ is GNS-faithful. This is done exactly as in the proof of assertion $(2)$ of \cite[Proposition 2.8]{FG15}.

\vspace{0.2cm}

\noindent$(3)$. The proof is the same as the proof of assertion $(3)$ of \cite[Proposition 2.8]{FG15}, by using the universal property stated in Proposition \ref{Prop-pvertex} and the definition of $P_{{\rm vert}}$.
\end{proof}
\noindent\textit{Notation.} We sometimes write $\EE_{A_p}^\Gr=\EE_{A_p}$.

\begin{proposition}\label{Cor-vertId}
Let $\Gr'\subset \Gr$ be a connected subgraph with maximal subtree $\mathcal{T}'\subset\mathcal{T}$. There exists a unique faithful $*$-homomorphism $\pi_{{\rm vert}}^{\Gr'}\,:\, P_{{\rm vert}}^{\Gr'}\rightarrow P_{{\rm vert}}$ such that $\pi_{{\rm vert}}^{\Gr'}\circ\lambda_{\Gr'}=\lambda\circ\pi_{\Gr'}$. The morphism $\pi_{{\rm vert}}^{\Gr'}$ satisfies $\EE_p\circ\pi_p^{\Gr'}=\EE_p^{\Gr'}$ for all $p\in V(\Gr)$. Moreover, there exists a unique ucp map $\EE_{\Gr'}\,:\,P_{{\rm vert}}\rightarrow P_{{\rm vert}}^{\Gr'}$ such that $\lambda_{v,p}^{\Gr'}\circ \EE_{\Gr'}=\EE_p^{\Gr'}\circ\lambda_{v,p}$ for all $p\in V(\Gr')$.
\end{proposition}

\begin{proof}
Define $P'=\lambda\circ\pi_{\Gr'}(P_{\Gr'})\subset P_{{\rm vert}}$ and consider, for $p\in V(\Gr)$, the ucp map $E_{A_p}=\EE_{A_p}\vert_{P'}$. Using the universal property of Proposition \ref{Prop-vertex}, assertion $3$, it suffices to check that the family $\{E_{A_p}\,:\,p\in V(\Gr)\}$ is GNS-faithful. Let $x\in P'$ such that $E_{A_p}(y^*x^*xy)=0$ for all $y\in P'$ and all $p\in V(\Gr)$. Arguing as in the proof of Proposition \ref{Cor-Univ-pred} we find that $\EE_{A_p}(y^*x^*xy)=0$ for all $y\in P_{{\rm vert}}$ and all $p\in V(\Gr)$. Since the family $\{\EE_{A_p}\,:\,p\in V(\Gr)\}$ is GNS faithful, the family $\{E_{A_p}\,:\,p\in V(\Gr)\}$ is also GNS-faithful. The construction of the canonical ucp map $\EE_{\Gr'}\,:\,P_{{\rm vert}}\rightarrow P_{{\rm vert}}^{\Gr'}$ is similar to the construction made in the proof of Proposition \ref{Cor-Univ-pred}. Indeed, let $A=\bigoplus_{p\in V(\Gr)} A_p$ and consider the Hilbert $A$-module $\bigoplus_{p\in V(\Gr)}H_{p,p}$ with the (faithful) left action of $P_{{\rm vert}}$ given by $\nu=\bigoplus_{p\in V(\Gr)}\lambda_{v,p}$. As in the proof of Proposition \ref{Cor-Univ-pred}, given any $p\in V(\Gr')$, we identify the Hilbert module of path in $\Gr'$ from $p$ to $p$, with the canonical Hilbert $A_p$-submodule $H'_{p,p}\subset H_{p,p}$ and we also view $\bigoplus_{p\in V(\Gr')}H'_{p,p}\subset \bigoplus_{p\in V(\Gr)}H_{p,p}$ as a Hilbert $A$-submodule. Note that the left action $\bigoplus_{p\in V(\Gr')}\lambda_{v,p}^{\Gr'}$ of $P_{{\rm vert}}^{\Gr'}$ on $\bigoplus_{p\in V(\Gr')}H'_{p,p}$ is faithful so that we may and will view $P_{{\rm vert}}^{\Gr'}\subset\mathcal{L}_A(\bigoplus_{p\in V(\Gr')}H'_{p,p})$. Let $Q\in\mathcal{L}_A(\bigoplus_{p\in V(\Gr)}H_{p,p})$ be the orthogonal projection onto $\bigoplus_{p\in V(\Gr')}H'_{p,p}$. Then it is not difficult to check that the ucp map $x\mapsto Q\nu(x)Q$ has the desired properties.
\end{proof}

\begin{example}\label{ExHNN}
When the graph contains two edges: $e$ and its opposite $\overline{e}$ then either $s(e)\neq r(e)$ and the construction considered above is the vertex-reduced amalgamated free product studied in \cite[Section 2]{FG15} or $s(e)=r(e)$ and the construction above is the vertex-reduced HNN-extension. Let us reformulate in details below our construction in that specific case. Note that the \textit{edge-reduced HNN-extension} has been described in details in \cite{Fi13}.

\noindent Let $A,B$ be unital C*-algebras and, for $\epsilon\in\{-1,1\}$, let $\pi_\epsilon\,:\,B\rightarrow A$ be a unital faithful $*$-homomorphism and $E_\epsilon\,:\, A\rightarrow B$ a ucp map such that $E_\epsilon\circ\pi_\epsilon=\id_B$. The full HNN-extension is the universal unital C*-algebra generated by $A$ and a unitary $u$ such that $u\pi_{-1}(b)u^*=\pi_1(b)$ for all $b\in B$. We denote this C*-algebra by ${\rm HNN}(A,B,\pi_1,\pi_{-1})$. The (vertex) reduced HNN-extension $C$ is the unique, up to isomorphism, unital C*-algebra satisfying the following properties:
\begin{enumerate}
\item There exists a unital $*$-homomorphism $\rho\,:\, A\rightarrow C$ and a unitary $u\in C$ such that $u\rho(\pi_{-1}(b))u^*=\rho(\pi_1(b))$ for all $b\in B$ and $C$ is generated by $\rho(A)$ and $u$.
\item There exists a GNS-faithful ucp map $E\,:\, C\rightarrow A$ such that $E\circ\rho=\id_A$ and $E(x)=0$ for all $x\in C$ of the form $x=\rho(a_0)u^{\epsilon_1}\dots u_{\epsilon_n}\rho(a_n)$ where $n\geq 1$, $a_k\in A$ and $\epsilon_k\in\{-1,1\}$ are such that, for all $1\leq k\leq n-1$, $\epsilon_{k+1}=-\epsilon_k\implies E_{-\epsilon_k}(a_k)=0$.
\item If $D$ is a unital C*-algebra with a unital $*$-homomorphism $\nu\,:\, A\rightarrow D$, a unitary $v\in D$ and a GNS-faithful ucp map $E'\,:\, D\rightarrow A$ such that
\begin{itemize}
\item $v\nu(\pi_{-1}(b))v^*=\nu(\pi_1(b))$ for all $b\in B$ and $D$ is generated by $\nu(A)$ and $v$.
\item $E'\circ\nu=\id_A$ and $E'(x)=0$ for all $x\in D$ of the form $x=\nu(a_0)v^{\epsilon_1}\dots v^{\epsilon_n}\nu(a_n)$ with $n\geq 1$, $\epsilon_k\in\{-1,1\}$, $a_k\in A$ such that, for all $1\leq k\leq n-1$ one has $\epsilon_{k+1}=-\epsilon_k\implies E_{-\epsilon_k}(a_k)=0$.
\end{itemize}
Then there exists a unique unital $*$-homomorphism $\widetilde{\nu}\,:\, C\rightarrow D$ such that $\widetilde{\nu}\circ\rho=\nu$ and $\widetilde{\nu}(u)=v$. Moreover, $E'\circ\widetilde{\nu}=E$. We denote this C*-algebra by ${\rm HNN}_{\rm vert}(A,B,\pi_1,\pi_{-1})$
\end{enumerate}
\end{example}

\noindent We now describe the \textit{Serre's devissage} process for our vertex-reduced fundamental C*-algebras.

\vspace{0.2cm}

\noindent For $e\in E(\Gr)$, let $\Gr_e$ be the graph obtained from $\Gr$ by removing the edges $e$ and $\overline{e}$ i.e., $V(\Gr_e)=V(\Gr)$ and $E(\Gr_e)=E(\Gr)\setminus\{e,\overline{e}\}$. The source range and inverse maps are the restrictions of the one for $\Gr$. The \textit{Serre's devissage} shows that, when $\Gr_e$ is not connected, the vertex-reduced fundamental C*-algebra is a vertex-reduced amalgamated free product and, when $\Gr_e$ is connected, the vertex-reduced fundamental C*-algebra is a vertex-reduced HNN-extension. We shall use freely the notations and results of \cite[Section 2]{FG15} about vertex-reduced amalgamated free products.

\vspace{0.2cm}

\noindent\textit{Case 1: $\Gr_e$ is not connected.} Let $\Gr_{s(e)}$ (respectively $\Gr_{r(e)}$) the connected component of $s(e)$ (resp. $r(e)$) in $\Gr_e$. Since $\Gr_e$ is not connected $e\in E(\mathcal{T})$ and the graphs $\mathcal{T}_{s(e)}:=\mathcal{T}\cap\Gr_{s(e)}$ and $\mathcal{T}_{r(e)}:=\mathcal{T}\cap\Gr_{r(e)}$ are maximal subtrees of $\Gr_{s(e)}$ and $\Gr_{r(e)}$ respectively. Let $P_{\Gr_{s(e)}}$ and $P_{\Gr_{r(e)}}$ be the maximal fundamental C*-algebras of our graph of C*-algebras restricted to $\Gr_{s(e)}$ and $\Gr_{r(e)}$ respectively and with respect to the maximal subtrees $\mathcal{T}_{s(e)}$ and $\mathcal{T}_{r(e)}$ respectively. Recall that we have canonical maps $\pi_{\Gr_{s(e)}}\,:\, P_{\Gr_{s(e)}}\rightarrow P$ and $\pi_{\Gr_{r(e)}}\,:\, P_{\Gr_{r(e)}}\rightarrow P$.

\vspace{0.2cm}

\noindent Let $P_{\Gr_{s(e)}}\underset{B_e}{*} P_{\Gr_{r(e)}}$ be the full free product of $P_{\Gr_{s(e)}}$ and $P_{\Gr_{r(e)}}$ amalgamated over $B_e$ relative to the maps $s_e\,:\, B_e\rightarrow P_{\Gr_{s(e)}}$ and $r_e\,:\,B_e\rightarrow P_{\Gr_{r(e)}}$. Observe that, since $e\in E(\mathcal{T})$, we have $u_e=1\in P$. Hence, we have $s_e(b)=r_e(b)$ in $P$, for all $b\in B_e$. By the universal property of the full amalgamated free product there exists a unique unital $*$-homomorphism $\nu\,:\,P_{\Gr_{s(e)}}\underset{B_e}{*} P_{\Gr_{r(e)}}\rightarrow P$ such that $\nu\vert_{P_{\Gr_{s(e)}}}=\pi_{\Gr_{s(e)}}$ and $\nu\vert_{P_{\Gr_{r(e)}}}=\pi_{\Gr_{r(e)}}$. Moreover, by the universal property of $P$, there exists also a unital $*$-homomorphism $P\rightarrow P_{\Gr_{s(e)}}\underset{B_e}{*} P_{\Gr_{r(e)}}$ which is the inverse of $\nu$. Hence, $\nu$ is a $*$-isomorphism. Actually, this is also true at the vertex-reduced level.

\noindent Note that we have injective unital $*$-homomorphisms $\iota_{s(e)}=\lambda_{\Gr_{s(e)}}\circ s_e\,:\,B_e\rightarrow P_{{\rm vert}}^{\Gr_{s(e)}}$ and $\iota_{s(e)}=\lambda_{\Gr_{r(e)}}\circ r_e\,:\,B_e\rightarrow P_{{\rm vert}}^{\Gr_{r(e)}}$ and conditional expectations $E_{s(e)}=\lambda_{\Gr_{s(e)}}\circ E_e^s\circ\EE^{\Gr_{s(e)}}_{A_{s(e)}}$ from $P_{{\rm vert}}^{\Gr_{s(e)}}$ to $\iota_{s(e)}(B_e)$ and $E_{r(e)}=\lambda_{\Gr_{r(e)}}\circ E_e^r\circ\EE_{A_{r(e)}}^{\Gr_{r(e)}}$ from $P_{{\rm vert}}^{\Gr_{r(e)}}$ to $\iota_{r(e)}(B_e)$ so that we can perform the vertex-reduced amalgamated free product. Following \cite[Section 2]{FG15}, we write
$$\pi\,:\,P_{{\rm vert}}^{\Gr_{s(e)}}\underset{B_e}{*} P_{{\rm vert}}^{\Gr_{r(e)}}\rightarrow P_{{\rm vert}}^{\Gr_{s(e)}}\underset{B_e}{\overset{v}{*}} P_{{\rm vert}}^{\Gr_{r(e)}}$$
the canonical surjection for the full amalgamated free product to the vertex-reduced amalgamated free product and $\EE_1$ (resp. $\EE_2$) the canonical ucp map from $P_{{\rm vert}}^{\Gr_{s(e)}}\underset{B_e}{\overset{v}{*}} P_{{\rm vert}}^{\Gr_{r(e)}}$ to $ P_{{\rm vert}}^{\Gr_{s(e)}}$ (resp. to $P_{{\rm vert}}^{\Gr_{r(e)}}$).

\begin{lemma}\label{DevissageFree}
There exists a unique  $*$-isomorphism $\nu_e\,:\,P_{{\rm vert}}^{\Gr_{s(e)}}\underset{B_e}{\overset{v}{*}} P_{{\rm vert}}^{\Gr_{r(e)}}\rightarrow P_{{\rm vert}}$ such that:
$$\nu_e\circ\pi\circ\lambda_{\Gr_{s(e)}}=\lambda\circ\pi_{\Gr_{s(e)}}\quad\text{and}\quad\nu_e\circ\pi\circ\lambda_{\Gr_{r(e)}}=\lambda\circ\pi_{\Gr_{r(e)}}.$$
Moreover we have $\EE_{\Gr_{s(e)}}\circ\nu_e=\EE_1$ and  $\EE_{\Gr_{r(e)}}\circ\nu_e=\EE_2$.

\end{lemma}

\begin{proof}
The proof is the same as the proof of \cite[Lemma 3.26]{FF13}, it suffices to prove that $P_{{\rm vert}}$ satisfies the universal property of $P_{{\rm vert}}^{\Gr_{s(e)}}\underset{B_e}{\overset{v}{*}} P_{{\rm vert}}^{\Gr_{r(e)}}$: the canonical ucp maps from $P_{{\rm vert}}$ to $P_{{\rm vert}}^{\Gr_{s(e)}}$ and $P_{{\rm vert}}^{\Gr_{r(e)}}$ are the ones constructed in Proposition \ref{Cor-vertId} i.e. $\EE_{\Gr_{s(e)}}$ and $\EE_{\Gr_{r(e)}}$. By \cite[Proposition 2.8, assertion 3]{FG15}, the resulting isomorphism $\nu_e$ intertwines the canonical ucp maps.
\end{proof}

\noindent\textit{Case 2: $\Gr_e$ is connected.}

\vspace{0.2cm}

\noindent Let $e\in E(\Gr)$ and suppose that $\Gr_e$ is connected. Up to a canonical isomorphism of $P$ we may and will assume that $\mathcal{T}\subset\Gr_e$. So that we have the canonical unital $*$-homomorphism $\pi_{\Gr_e}\,:\, P_{\Gr_e}\rightarrow P$. We consider the two unital faithful $*$-homomorphisms $s_e,r_e\,:\, B_e\rightarrow P_{\Gr_e}$. By definition, we have $u_er_e(b)u_e^*=s_e(b)$ for all $b\in B_e$ and $P$ is generated, as a C*-algebra, by $\pi_{\Gr_e}(P_{\Gr_e})$ and $u_e$. By the universal property of the maximal HNN-extension, there exists a unique unital (surjective) $*$-homomorphism $\nu\,:\,{\rm HNN}(P_{\Gr_e},B_e,s_e,r_e)\rightarrow P$ such that $\nu\vert_{P_{\Gr_e}}=\pi_{\Gr_e}$ and $\nu(u)=u_e$. Observe that, by the universal property of $P$, there exists a unital $*$-homomorphism $P\rightarrow {\rm HNN}(P_{\Gr_e},B_e,s_e,r_e)$ which is the inverse of $\nu$. Hence $\nu$ is a $*$-isomorphism. Actually this is also true at the vertex-reduced level.

\vspace{0.2cm}

\noindent Define the faithful unital $*$-homomorphism $\pi_1,\pi_{-1}\,:\,B_e\rightarrow P^{\Gr_e}_{{\rm vert}}$ by $\pi_{-1}=\lambda_{\Gr_e}\circ s_e$ and $\pi_1=\lambda_{\Gr_e}\circ r_e$. Note that the ucp maps $E_{\epsilon}\,:\, P_{{\rm vert}}^{\Gr_e}\rightarrow B_e$ defined by $E_1=s_e^{-1}\circ E_e^s\circ\EE_{s(e)}^{\Gr_e}$ and $E_{-1}=r_e^{-1}\circ E_e^r\circ\EE_{r(e)}^{\Gr_e}$ satisfy $E_\epsilon\circ\pi_{\epsilon}=\id_{B_e}$ for $\epsilon\in\{-1,1\}$. Hence we may consider the vertex-reduced HNN-extension and the canonical surjection $\lambda_e\,:\,{\rm HNN}(P^{\Gr_e}_{{\rm vert}},B_e,s_e,r_e)\rightarrow  {\rm HNN}_{{\rm vert}}(P_{{\rm vert}}^{\Gr_e},B_e,\pi_1,\pi_{-1})$. Write $v=\lambda_e(u)$, where $u\in {\rm HNN}(P^{\Gr_e}_{{\rm vert}},B_e,s_e,r_e)$ is the "stable letter". Recall that, by Proposition \ref{Cor-vertId}, we have the canonical faithful unital $*$-homomorphism $\pi_{{\rm vert}}^{\Gr_e}\,:\,P_{{\rm vert}}^{\Gr_e}\rightarrow P_{{\rm vert}}$. Let $\EE\,:\,{\rm HNN}_{{\rm vert}}(P_{{\rm vert}}^{\Gr_e},B_e,\pi_1,\pi_{-1})\rightarrow P_{{\rm vert}}^{\Gr_e}$ the canonical GNS-faithful ucp map.

\begin{lemma}\label{DevissageHNN}
There is a unique $*$-isomorphism $\nu_e\,:\, {\rm HNN}_{{\rm vert}}(P_{{\rm vert}}^{\Gr_e},B_e,\pi_1,\pi_{-1})\rightarrow P_{{\rm vert}}$ such that $\nu_e\circ\lambda_e\vert_{P_{{\rm vert}}^{\Gr_e}}=\pi_{{\rm vert}}^{\Gr_e}$ and $\nu_e(u)=u_e$. Moreover $\EE_{\Gr_e}\circ\nu_e=\EE$.
\end{lemma}

\begin{proof}
Since we have $u_e\pi_{{\rm vert}}^{\Gr_e}(\pi_{-1}(b)) u_e^*=\pi_{{\rm vert}}^{\Gr_e}(\pi_1(b))$ for all $b\in B_e$, it suffices, by the universal property of the vertex-reduced HNN-extension explained in Example \ref{ExHNN}, to check that we have a GNS-faithful ucp map $P_{{\rm vert}}\rightarrow P_{{\rm vert}}^{\Gr_e}$ satisfying the conditions described in Example \ref{ExHNN}. This ucp map is the one constructed in Proposition \ref{Cor-vertId}: it is the map $\EE_{\Gr_e}$ and the conditions can be checked as in the proof of \cite[Lemma 3.27]{FF13}. The fact that the resulting isomorphism $\nu_e$ intertwines the ucp maps follows from the universal property.
\end{proof}
\noindent We end this preliminary section with an easy Lemma.

\begin{lemma}\label{LemNorm}
If $x=a_0u_{e_1}\dots u_{e_n}a_n\in P$ is a reduced operator from $p$ to $p$ and $a_n\in B_{e_n}^r$ then
$$\EE_{p}(\lambda_p(x^*x))=\EE_{e_n}^r\circ\EE_{p}(\lambda_p(x^*x)).$$
\end{lemma}

\begin{proof}
Define $x_0=a_0^*a_0$ and for $1\leq k\leq n$, $x_k=a_k^*(r_{e_k}\circ s_{e_k}^{-1}\circ\EE_{e_k}^s(x_{k-1}))a_k$. We apply Lemma \ref{LemProduct} to the pair $a=x^*$ and $b=x$ in case $(2)$. It follows that $x^*x=y+x_n$, where $y$ is a sum of reduced operators from $p$ to $p$. Hence $\EE_{p}(\lambda_p(y))=0$ and, since $a_n\in B_{e_n}^r$, we have $x_n=a_n^*(r_{e_n}\circ s_{e_n}^{-1}\circ\EE_{e_n}^s(x_{n-1}))a_n\in B_{e_n}^r$.
\end{proof}

%%%%%%%%%%%%%%%%%%%%%%%%%%%%%%%%%%%%%%%%%%
\section{Boundary maps}
%%%%%%%%%%%%%%%%%%%%%%%%%%%%%%%%%%%%%%%%%%

\noindent Define the ucp map $\EE_e=E_e^r\circ\EE_{A_{r(e)}}\,:\, P_{{\rm vert}}\rightarrow B_e^r$. Note that the GNS construction of $\EE_e$ is given by $(H_{r(e),r(e)}\underset{E_e^r}{\ot}B_e^r,\lambda_{v,r(e)}\ot 1,\xi_{r(e)}\ot 1)$. To simplify the notations, we will denote by $(K_e,\rho_e,\eta_e)$ the GNS construction of $\EE_e$. We define $\mathcal{R}_{e}\subset K_e$ as the Hilbert $B_e^r$-submodule of $K_e$ of the ``words ending with $e$". More precisely,
$$\mathcal{R}_{e}:=\overline{\text{Span}}\{\rho_e(\lambda(x))\eta_e\,\vert\,x=a_0u_{e_1}\dots u_{e_n}a_n\in P\text{ reduced from $r(e)$ to $r(e)$}$$
$$\text{ with }e_n=e\text{ and }a_n\in B_e^r\}\subset K_e.$$
It is easy to see from the definition that $\mathcal{R}_{e}$ is a Hilbert $B_e^r$-submodule of $K_e$. Moreover, it is complemented in $K_e$ with orthogonal complement given by:
$$\mathcal{L}_{e}:=\overline{\text{Span}}\{\rho_e(\lambda(x))\eta_e\,\vert\,x\in A_{r(e)}\text{ or }x=a_0u_{e_1}\dots u_{e_n}a_n\in P\text{ reduced from $r(e)$ to $r(e)$ with }$$
$$e_n\neq e\text{ or }e_n=e\text{ and }a_n\in A_{r(e)}\ominus B_e^r\}.$$
Let $Q_e\in\mathcal{L}_{B_e^r}(K_e)$ be the orthogonal projection onto $\mathcal{R}_{e}$ and define 
$$X_e=\{x=a_0u_{e_1}\dots u_{e_n}a_n\in P\text{ reduced from $r(e)$ to $r(e)$ with }e_k\notin\{\overline{e},e\}\text{ for all }1\leq k\leq n\},$$

\begin{lemma}\label{LemKK1}
The following holds.
\begin{enumerate}
\item For all reduced operator $a=a_nu_{e_n}\dots u_{e_1}a_0\in P$ from $r(e)$ to $r(e)$ we have
$${\rm Im}(Q_e\rho_e(\lambda(a))-\rho_e(\lambda(a))Q_e)\subset \overline{X_a},\text{ where:}$$
$$X_a=\left\{\begin{array}{ll}Y_a:=\left(\underset{k\in\{1,\dots,n\}, e_k=e}{\bigoplus}\rho_e(\lambda(a_nu_{e_n}\dots u_{e_{k}}))\eta_e\cdot B_e^r\right)&\text{if }e\text{ is not a loop,}\\
Y_a\oplus\left(\underset{k\in\{1,\dots,n\}, e_{k}=\overline{e}}{\bigoplus}\rho_e(\lambda(a_nu_{e_n}\dots u_{e_{k+1}}a_{k}))\eta_e\cdot B_e^r\right)&\text{if }e\text{ is a loop.}\end{array}\right.$$
(by convention, the term in the last direct sum is $\rho_e(\lambda(a_n))\eta_e\cdot B_e^r$, when $e_n=\overline{e}$ is a loop)
\item $Q_e$ commutes with $\rho_e(\lambda(a))$ for all $a\in \overline{{\rm Span}}\left(A_{r(e)}\cup X_e\right)$.
\item $Q_e\rho_e(\lambda(a))-\rho_e(\lambda(a))Q_e\in\mathcal{K}_{B_e^r}(K_e)$ for all $a\in P$.
\end{enumerate}
\end{lemma}

\begin{proof}
During the proof we will use the notation $\mathcal{P}_e^r(x)=x-E_e^r(x)$ for $x\in A_{r(e)}$.

\vspace{0.2cm}

\noindent$(1)$. Let $n\geq 1$ and $a=a_nu_{e_n}\dots u_{e_1}a_0\in P$ a reduced operator from $r(e)$ to $r(e)$.

\noindent Suppose that $b\in A_{r(e)}$. We have $Q_e\rho_e(\lambda(b))\eta_e=0$ and $ab=a_nu_{e_n}\dots u_{e_1}a_0b\in P$ is reduced. Hence, if $e_1\neq e$, we have $Q_e\rho_e(\lambda(ab))\eta_e=0$ and, if $e_1= e$, we have
$$ab=a_nu_{e_n}\dots u_{e}E_e^r(x_0)+a_n\dots u_{e}\mathcal{P}_e^r(x_0)\quad\text{where}\quad x_0=a_0b.$$
It follows that $Q_e\rho_e(\lambda(ab))\eta_e=\rho_e(\lambda(a_nu_{e_n}\dots u_{e}E_e^r(x_0)))\eta_e$. To conclude we have, $\forall b\in A_{r(e)}$,
$$(Q_e\rho_e(\lambda(a))-\rho_e(\lambda(a))Q_e)\rho_e(\lambda(b))\eta_e=\left\{\begin{array}{lcl}0\in X_a&\text{if}&e_1\neq e,\\\rho_e(\lambda(a_nu_{e_n}\dots u_{e_1}))\eta_e\cdot E_e^r(a_0b)\in X_a&\text{if}&e_1=e.\end{array}\right.$$

\noindent Suppose that $b=b_0u_{f_1}\dots u_{f_m}b_m\in P$ is a reduced operator from $r(e)$ to $r(e)$. Let $0\leq n_0\leq{\rm min}\{n,m\}$ be the integer associated to the couple $(a,b)$ in Lemma \ref{LemProduct}. This Lemma implies that, when $n_0=0$ or $n_0=n<m$ or $1\leq n_0<{\rm min}\{n,m\}$, $ab$ is a reduced word or a sum of reduced words that end by $u_{f_m}b_m$. Hence, in this cases, we have $\rho_e(\lambda(b))\eta_e\in\mathcal{R}_{e}\implies\rho_e(\lambda(ab))\eta_e\in\mathcal{R}_{e}$ and  $\rho_e(\lambda(b))\eta_e\in\mathcal{L}_{e}\ominus A_{r(e)}\implies\rho_e(\lambda(ab))\eta_e\in\mathcal{L}_{e}\ominus A_{r(e)}$. It follows that $(Q_e\rho_e(\lambda(a))-\rho_e(\lambda(a))Q_e)\rho_e(\lambda(b))\eta_e=0\in X_a$.

\vspace{0.2cm}

\noindent Suppose now that $n_0=m<n$. Lemma \ref{LemProduct} implies that $ab=y+z$ where $y$ is a sum of reduced words that end by $u_{f_m}b_m$ and $z=a_nu_{e_n}\dots u_{e_{m+1}}x_m$. Hence we have $\rho_e(\lambda(b))\eta_e\in\mathcal{R}_{e}\implies\rho_e(\lambda(y))\eta_e\in\mathcal{R}_{e}$ and  $\rho_e(\lambda(b))\eta_e\in\mathcal{L}_{e}\ominus A_{r(e)}\implies\rho_e(\lambda(y))\eta_e\in\mathcal{L}_{e}\ominus A_{r(e)}$. It follows that
$$(Q_e\rho_e(\lambda(a))-\rho_e(\lambda(a))Q_e)\rho_e(\lambda(b))\eta_e=\left\{\begin{array}{lcl}
Q_e\rho_e(\lambda(z))\eta_e&\text{if}&\rho_e(\lambda(b))\eta_e\in\mathcal{L}_{e},\\
Q_e\rho_e(\lambda(z))\eta_e-\rho_e(\lambda(z))\eta_e&\text{if}&\rho_e(\lambda(b))\eta_e\in\mathcal{R}_{e}.\end{array}\right.$$
We have $\rho_e(\lambda(z))\eta_e=\rho_e(\lambda(a_nu_{e_n}\dots u_{e_{m+1}}x_m))\eta_e$ hence,
$$Q_e\rho_e(\lambda(z))\eta_e=\left\{\begin{array}{lcl}
0\in X_a&\text{if}&e_{m+1}\neq e\text{ or }e_{m+1}=e\text{ and }x_m\in A_{r(e)}\ominus B_e^r,\\
\rho_e(\lambda(a_nu_{e_n}\dots u_{e_{m+1}}))\eta_e. x_m\in X_a&\text{if} &e_{m+1}=e\text{ and }x_m\in  B_e^r.\end{array}\right.$$
Hence $(Q_e\rho_e(\lambda(a))-\rho_e(\lambda(a))Q_e)\rho_e(\lambda(b))\eta_e\in X_a$ if $\rho_e(\lambda(b))\eta_e\in\mathcal{L}_{e}$ and, if $\rho_e(\lambda(b))\eta_e\in\mathcal{R}_{e}$, we have $f_m=e$ and $b_m\in B_e^r$. Since $n_0=m$ we conclude that $e_m=\overline{f}_m=\overline{e}$ and $x_m$ is equal to $a_m(r_e\circ s_e^{-1}\circ E_e^s(x_{m-1}))b_m$. Note that, since $r(f_m)=r(e)$ and $f_m=\overline{e}$ we find that $s(e)=r(f_m)=r(e)$. Hence $e$ must be a loop. Moreover, $\rho_e(\lambda(z))\eta_e=\rho_e(\lambda(a_nu_{e_n}\dots u_{e_{m+1}}a_m))\eta_e\cdot x_m'\in X_a$, where $x_m'=(r_e\circ s_e^{-1}\circ E_e^s(x_{m-1}))b_m\in B_e^r$. It follows that $(Q_e\rho_e(\lambda(a))-\rho_e(\lambda(a))Q_e)\rho_e(\lambda(b))\eta_e\in X_a$ also when $\rho_e(\lambda(b))\eta_e\in\mathcal{R}_{e}$.

\vspace{0.2cm}

\noindent Suppose that $n_0=n=m$. Lemma \ref{LemProduct} implies that $ab=y+x_m$ where $y$ is a sum of reduced words that end by $u_{f_m}b_m$. As before, we deduce that:
$$(Q_e\rho_e(\lambda(a))-\rho_e(\lambda(a))Q_e)\rho_e(\lambda(b))\eta_e=\left\{\begin{array}{lcl}
Q_e\rho_e(\lambda(x_m))\eta_e=0&\text{if}&\rho_e(\lambda(b))\eta_e\in\mathcal{L}_{e},\\
Q_e\rho_e(\lambda(x_m))\eta_e-\rho_e(\lambda(x_m))\eta_e&\text{if}&\rho_e(\lambda(b))\eta_e\in\mathcal{R}_{e}.\end{array}\right.$$
And, if $\rho_e(\lambda(b))\eta_e\in\mathcal{R}_{e}$ then $f_m=e$ and $b_m\in B_e^r$. Since $n_0=m=n$, we deduce that $e_n=\overline{f_m}=\overline{e}$ (hence $e$ is a loop) and $x_m=a_n(r_e\circ s_e^{-1}\circ E_e^s(x_{n-1}))b_n\in a_n B_e^r$. Hence,
$$Q_e\rho_e(\lambda(x_m))\eta_e-\rho_e(\lambda(x_m))\eta_e=-\rho_e(\lambda(x_m))\eta_e=-\rho_e(\lambda(a_n))\eta_e\cdot x_n'\in X_a,$$
where $x_n'=(r_e\circ s_e^{-1}\circ E_e^s(x_{n-1}))b_n\in B_e^r$. This concludes the proof of the Lemma.

\vspace{0.2cm}

\noindent$(2)$. It is obvious that $\rho_e(\lambda(a))$ commutes with $Q_e$ for all $a\in A_{r(e)}$. Hence, $(2)$ follows from $(1)$.

\vspace{0.2cm}

\noindent$(3)$. Again, it directly follows from the computations made in $(1)$ but we write the details for the convenience of the reader. Since any reduced operator in $P$ from $r(e)$ to $r(e)$ may be written as a product of reduced operators $a\in P$ from $r(e)$ to $r(e)$ of the form $(I)$: the edges in $a$ are all different from $e$ or $\overline{e}$; $(II)$: $a=u_{\overline{e}}x$, where $x$ is a reduced operator from $s(e)$ to $r(e)$ whose edges are all different from $e$ or $\overline{e}$; $(III)$: $a=xu_{e}$, where $x$ is a reduced operator from $r(e)$ to $s(e)$ whose edges are all different from $e$ or $\overline{e}$. By $(2)$ $\rho_e(\lambda(a))$ commutes with $Q_e$ for $a$ of type $(I)$ and, since any element of type $(II)$ is the adjoint of an element of type$(III)$, it suffices to show that the commutator of $Q_e$ and $\rho_e(\lambda(a))$ is compact for all $a$ of type $(III)$. First assume that $e$ is a loop. In that case, it suffices to show that $Q_e\rho_e(\lambda(u_e))-\rho_e(\lambda(u_e))Q_e$ is compact. Let $b\in P$. From the computations made in $(1)$, we see that $(Q_e\rho_e(\lambda(u_e))-\rho_e(\lambda(u_e))Q_e)\rho_e(\lambda(b))\eta_e=0$ for any $b\in P$ reduced operator from $r(e)$ to $r(e)$ and, for $b\in A_{r(e)}$ one has
$$(Q_e\rho_e(\lambda(u_e))-\rho_e(\lambda(u_e))Q_e)\rho_e(\lambda(b))\eta_e=\rho_e(\lambda(u_e)\eta_e\cdot E_e^r(b)=\rho_e(\lambda(u_e))\eta_e\cdot\langle\eta_e,\rho_e(\lambda(b))\eta_e\rangle.$$
Hence, the equality $(Q_e\rho_e(\lambda(u_e))-\rho_e(\lambda(u_e))Q_e)\xi=\rho_e(\lambda(u_e))\eta_e\cdot\langle\eta_e,\xi\rangle$ holds for any $\xi=\rho_e(\lambda(b))\eta_e$ with $b$ in the span of $A_{r(e)}$ and the reduced operator in $P$ from $r(e)$ to $r(e)$. Hence, it holds for any $\xi\in K_e$. It follows that the commutator of $Q_e$ and $\rho_e(\lambda(u_e))$ is a rank one operator, hence compact. Let us now assume that $e$ is not a loop. Write $a=a_nu_{e_n}\dots u_{e_1}a_0u_{e}$, where $n\geq 1$, $e_k\notin\{e,\overline{e}\}$ for all $k$. For $b\in P$ we write $X(b)=(Q_e\rho_e(\lambda(a))-\rho_e(\lambda(a))Q_e)\rho_e(\lambda(b))\eta_e$. As before, following the computations made in $(1)$ we see that, since $e_k\notin\{e,\overline{e}\}$, we have $X(b)=0$ whenever $b$ is a reduced operator from $r(e)$ to $r(e)$. Moreover, when $b\in A_{r(e)}$ we have $X(b)=\rho_e(\lambda(a))\eta_e\cdot\langle\eta_e,\rho_e(\lambda(b))\eta_e\rangle$. As before, it follows that the commutator of $Q_e$ and $\rho_e(\lambda(a))$ is a rank one operator.

\end{proof}

\noindent Define $V_e=2Q_e-1\in\mathcal{L}_{B_e^r}(K_e)$. We have $V_e^2=1$, $V_e=V_e^*$ and, for all $x\in P_{{\rm vert}}$, Lemma \ref{LemKK1} implies that $V_e\rho_e(x)-\rho_e(x)V_e\in\mathcal{K}_{B_e^r}(K_e)$. Hence we get an element $y_e^\Gr\in KK^1(P_{{\rm vert}},B_e^r)$. Define $x_e^\Gr=y_e^{\Gr}\underset{B_e^r}{\ot}[r_e^{-1}]\in KK^1(P_{{\rm vert}},B_e)$.
\begin{remark}\label{RmkFullKK1-1}
Note that we also have an element $z_e^\Gr=[\lambda]\underset{P_{{\rm vert}}}{\ot} x_e^\Gr\in KK^1(P,B_e)$.\end{remark}
\noindent Recall that for a subgraph $\mathcal{G}'\subset\mathcal{G}$ with a maximal subtree $\T'\subset\Gr'$ such that $\T'\subset\T$ we have the canonical unital faithful $*$-homomorphism $\pi_{{\rm vert}}^{\Gr'}\,:\, P_{{\rm vert}}^{\Gr'}\rightarrow P_{{\rm vert}}$ defined in Proposition \ref{Cor-vertId}.

\begin{proposition}\label{PropKK1}
For all connected subgraph $\Gr'\subset \Gr$ with maximal subtree $\T'\subset\T$, we have 
\begin{enumerate}
\item if $e\in E(\Gr')$ then $[\pi_{{\rm vert}}^{\Gr'}]\underset{P_{{\rm vert}}}{\ot}x_e^{\Gr}= x_e^{\Gr'}\in KK^1(P_{{\rm vert}}^{\Gr'},B_e)$,
\item if $e\notin E(\Gr')$ then $[\pi_{{\rm vert}}^{\Gr'}]\underset{P_{{\rm vert}}}{\ot}x_e^{\Gr}=0\in KK^1(P_{{\rm vert}}^{\Gr'},B_e)$,
\item $\sum_{ r(e)=p}x_e^{\Gr}\underset{B_e}{\ot} [r_e]=0\in KK^1(P_{{\rm vert}},A_{p})$ for all $p\in V(\Gr)$,
\item For all $e\in E(\Gr)$ we have $x_{\overline{e}}^{\Gr}=-x_e^{\Gr}$.
\end{enumerate}
\end{proposition}

\begin{proof}
Let $\Gr'\subset\Gr$ be a connected subgraph with maximal subtree $\T'\subset\T$ and $e\in E(\Gr)$.

\vspace{0.2cm}

\noindent $(1)$. Suppose that $e\in E(\Gr')$ (hence $\overline{e}\in E(\Gr')$). Recall that we have the canonical ucp map $\EE_{\Gr'}\,:\, P_{{\rm vert}}\rightarrow P_{{\rm vert}}^{\Gr'}$ from Proposition \ref{Cor-vertId}. Moreover, by definition of $\pi_{{\rm vert}}^{\Gr'}$ of we have $\EE_e^{\Gr'}=\EE_e\circ\pi_{{\rm vert}}^{\Gr'}$, where $\EE_e^{\Gr'}=E_e^r\circ\EE^{\Gr'}_{A_{r(e)}}$.

\noindent Let $(K_e,\rho_e,\eta_e)$ be the GNS construction of $\EE_{e}$ and define $K_e'=\overline{\rho_e\circ\pi_{{\rm vert}}^{\Gr'}(P_{{\rm vert}}^{\Gr'})\eta_e\cdot B_e^r}$. Observe that $K_e'$ is complemented. Indeed, we have $K_e'\oplus L_e=K_e$, where
$$L_e=\overline{\text{Span}}\{\rho_e(x)\eta_e\cdot b\,:\,b\in B_e^r\text{ and }x\in P_{{\rm vert}}\text{ such that }\EE_{\Gr'}(x)=0\}.$$
Let $R_e\in\mathcal{L}_{B_e^r}(K_e)$ be the orthogonal projection onto $K_e'$. Since $\rho_e\circ\pi_{{\rm vert}}^{\Gr'}(x)K_e'\subset K_e'$ for all $x\in P_{{\rm vert}}^{\Gr'}$, $R_e$ commutes with $\rho_e\circ\pi_{{\rm vert}}^{\Gr'}(x)$ for all $x\in P_{{\rm vert}}^{\Gr'}$. It is also easy to check that $R_e$ commutes with $Q_e$ hence with $V_e$.

\vspace{0.2cm}

\noindent Since $\EE_e^{\Gr'}=\EE_e\circ\pi_{{\rm vert}}^{\Gr'}$ the triple $(K_e',\rho_e',\eta_e')$, where $\rho_e'(x)=\rho_e\circ\pi_{{\rm vert}}^{\Gr'}(x)R_e$ for $x\in P_{{\rm vert}}^{\Gr'}$ and $\eta_e'=\eta_e$, is a GNS construction of $\EE^{\Gr'}_e$. Let $Q_e'\in\mathcal{L}_{B_e^r}(K_e')$ be the associated operator such that $x_e^{\Gr'}=[(K_e',\rho_e',V_e')]$, with $V_e'=2Q_e'-1$. By definition we have $Q_e'=Q_eR_e$ hence, $V_e'=V_eR_e$. It follows that $[\pi_{{\rm vert}}^{\Gr'}]\underset{P_{{\rm vert}}}{\ot} x_e^{\Gr}=x_e^{\Gr'}\oplus y$, where $y\in KK^1(P_{{\rm vert}}^{\Gr'},B_e)$ is represented by the triple $(L_e,\pi_e,V_e(1-R_e))$, where $\pi_e=\rho_e\circ\pi_{{\rm vert}}^{\Gr'}(\cdot)(1-R_e)$. To conclude the proof of $(1)$ it suffices to check that this triple is degenerated. Since $V_e$ and $(1-R_e)$ commute, $V_e(1-R_e)$ is self-adjoint and $(V_e(1-R_e))^2=1-R_e=\id_{L_e}$. Hence, it suffices to check that, for all $a\in P_{{\rm vert}}^{\Gr'}$,
$$(Q_e\rho_e\circ\pi_{{\rm vert}}^{\Gr'}(a)-\rho_e\circ\pi_{{\rm vert}}^{\Gr'}(a)Q_e)(1-R_e)=0.$$
We already know from assertion $(2)$ of Lemma \ref{LemKK1} that $Q_e\rho_e(\lambda(a))=\rho_e(\lambda(a))Q_e$ for all $a\in A_{r(e)}$ (and all $a\in X_e$). Let $a=a_nu_{e_n}\dots u_{e_1}a_0\in P_{\Gr'}$ and $b=b_0u_{f_1}\dots u_{f_m}b_m\in P$ be reduced operators from $r(e)$ to $r(e)$ and suppose that $\EE_{\Gr'}(\lambda(b))=0$. Hence, there exists $k\in\{1,\dots m\}$ such that $f_k\notin E(\Gr')$ and it follows that the integer $n_0$ associated to the pair $(\pi_{\Gr'}(a),b)$ in Lemma \ref{LemProduct} satisfies $n_0<k$ since $e_l\in E(\Gr')$ for all $l\in\{1,\dots n\}$. Applying Lemma \ref{LemProduct} in case $(5)$, we see that $\pi_{\Gr'}(a)b$ is a sum of reduced operators that end with $u_{f_m}b_m$. Hence, $\rho_e(\lambda(b))\eta_e\in\mathcal{R}_{e}\implies \rho_e(\lambda(\pi_{\Gr'}(a)b))\eta_e\in\mathcal{R}_{e}$ and $\rho_e(\lambda(b))\eta_e\in\mathcal{L}_{e}\implies \rho_e(\lambda(\pi_{\Gr'}(a)b))\eta_e\in\mathcal{L}_{e}$. It follows that
$$[Q_e\rho_e(\pi_{{\rm vert}}^{\Gr'}(\lambda_{\Gr'}(a)))-\rho_e(\pi_{{\rm vert}}^{\Gr'}(\lambda_{\Gr'}(a)))Q_e]\rho_e(\lambda(b))\eta_e$$
$$=[Q_e\rho_e(\lambda(\pi_{\Gr'}(a)))-\rho_e(\lambda(\pi_{\Gr'}(a)))Q_e]\rho_e(\lambda(b))\eta_e=0.$$
This concludes the proof of $(1)$.

\vspace{0.2cm}

\noindent $(2)$. Suppose that $e\notin E(\Gr')$  (hence $\overline{e}\notin E(\Gr')$). The element $[\pi_{{\rm vert}}^{\Gr'}]\underset{P_{{\rm vert}}}{\ot}x_e^{\Gr}$ is represented by the triple $(K_e,\pi_e,V_e)$, where $\pi_e=\rho_e\circ\pi_{{\rm vert}}^{\Gr'}$. Since $V_e^2=1$ and $V_e^*=V_e$, it suffices to show that $Q_e$ commutes with $\rho_e(\pi_{{\rm vert}}^{\Gr'}(x))$ for all $x\in P_{{\rm vert}}^{\Gr'}$. It follows from assertion $(2)$ of Lemma \ref{LemKK1} since $e,\overline{e}\notin E(\Gr')$ implies $\pi_{{\rm vert}}^{\Gr'}(P_{{\rm vert}}^{\Gr'})\subset\overline{{\rm Span}}\left(\lambda(A_{r(e)})\cup \lambda(X_e)\right)$.

\vspace{0.2cm}

\noindent $(3)$. For $p\in V(\Gr)$ we use the notation $(H_p,\pi_p,\xi_p):=(H_{p,p},\lambda_{v,p},\xi_p)$ for the GNS construction of the canonical ucp map $\EE_{A_p}\,:\,P_{{\rm vert}}\rightarrow A_p$. Observe that $\xi_p\cdot A_p$ is orthogonally complemented in $H_p$ and set $H_p^\circ=H_p\ominus\xi_p\cdot A_p$. Define $K_p=\bigoplus_{e\in E(\Gr),r(e)=p}K_e\underset{B_e^r}{\ot}A_p$ and observe that, by Lemma \ref{LemNorm}, we have an isometry $F_p\in\mathcal{L}_{A_p}(H_p^\circ,K_p)$ defined by
$$F_p(\pi_p(\lambda(a_0u_{e_1}\dots u_{e_n}a_n))\xi_p)=\rho_{e_n}(\lambda(a_0u_{e_1}\dots u_{e_n}))\eta_{e_n}\ot a_n,$$
for all $a_0u_{e_1}\dots u_{e_n}a_n\in P$ reduced operator from $p$ to $p$. We extend $F_p$ to partial isometry, still denoted $F_p\in\mathcal{L}_{A_p}(H_p,K_p)$ by $F_p\vert_{\xi_p\cdot A_p}=0$. Then $F_p^*F_p=1-Q_{\xi_p}$, where $Q_{\xi_p}\in\mathcal{L}_{A_p}(H_p)$ is the orthogonal projection onto $\xi_p\cdot A_p$. Moreover, $F_pF_p^*=\bigoplus_{e\in E(\Gr),r(e)=p}Q_e\ot 1$.

\vspace{0.2cm}

\noindent Define $\rho_p=\bigoplus_{e\in E(\Gr),r(e)=p}\rho_e\ot 1\,:\, P_{{\rm vert}}\rightarrow \mathcal{L}_{A_p}(K_p)$.

\begin{lemma}\label{LemFp}
For any $a\in P$ we have $(F_p\pi_p(\lambda(a))-\rho_p(\lambda(a))F_p)\in\mathcal{K}_{A_p}(H_p,K_p)$.
\end{lemma}

\begin{proof}
It suffices to prove the lemma for any $a=a_nu_{e_n}\dots u_{e_1}a_0\in P$ reduced operator from $p$ to $p$ since, for $a\in A_p$ one has $F_p\pi_p(\lambda(a))=\rho_p(\lambda(a))F_p$. We may and will assume that $r(e_k)\neq p$ for all $k\neq 1$ since reduced operator from $p$ to $p$ may be written as the product of such operators. Fix such an operator $a$ and, for $b\in P$ write $X(b)=(F_p\pi_p(\lambda(a))-\rho_p(\lambda(a))F_p)(\pi_p(\lambda(b))\xi_p)$. If $b\in A_p$ then $F_p\pi_p(\lambda(b))\xi_p=0$ and $ab=a_nu_{e_n}\dots u_{e_1}a_0b\in P$ is reduced from $p$ to $p$. Hence, $F_p\pi_p(\lambda(ab))\xi_p=\rho_{e_1}(\lambda(a_nu_{e_n}\dots u_{e_1}))\eta_{e_1}\ot a_0b$ and we have
\begin{eqnarray*}
X(b)&=&(\rho_{e_1}(\lambda(a_nu_{e_n}\dots u_{e_1}))\eta_{e_1}\ot 1)\cdot a_0b\\
&=&(\rho_{e_1}(\lambda(a_nu_{e_n}\dots u_{e_1}))\eta_{e_1}\ot 1)\cdot\langle\pi_p(\lambda(a_0^*))\xi_p,\pi_p(\lambda(b))\xi_p\rangle.
\end{eqnarray*}

\noindent Suppose that $b=b_0u_{f_1}\dots u_{f_m}b_m\in P$ is a reduced operator from $p$ to $p$ and write $b=b'b_m$, where $b'=b_0u_{f_1}\dots u_{f_m}$. Let $0\leq n_0\leq{\rm min}\{n,m\}$ and, for $1\leq k\leq n_0$, $x_k\in A_{s(e_k)}$ be the data associated to the couple $(a,b')$ in Lemma \ref{LemProduct}. By Lemma  \ref{LemProduct} we can write $ab'=y+z$, where $y$ is either reduced and ends with $u_{f_m}$ or is a sum of reduced operators that end with $u_{f_m}$ and:
$$z=\left\{\begin{array}{ll}a_nu_{e_n}\dots u_{e_{m+1}}x_m&\text{if }n_0=m<n,\\
x_n&\text{if }n_0=n=m,\\
0&\text{if }n_0=0\text{ or }n_0=n<m\text{ or }1\leq n_0<\text{min}\{n,m\}.\end{array}\right.$$
Since $y$ is a sum of reduced operators ending with $u_{f_m}$ we have $F_p\pi_p(\lambda(y))\xi_p=\rho_{f_m}(\lambda(y))\eta_{f_m}\ot 1$ and,
\begin{eqnarray*}
X(b)&=&F_p\pi_p(\lambda(ab'))\xi_p\cdot b_m-\rho_{f_m}(\lambda(ab'))\eta_{f_m}\ot b_m\\
&=&F_p\pi_p(\lambda(y))\xi_p\cdot b_m-\rho_{f_m}(\lambda(y))\eta_{f_m}\ot b_m+F_p\pi_p(\lambda(z))\xi_p\cdot b_m-\rho_{f_m}(\lambda(z))\eta_{f_m}\ot b_m\\
&=&F_p\pi_p(\lambda(z))\xi_p\cdot b_m-\rho_{f_m}(\lambda(z))\eta_{f_m}\ot b_m.
\end{eqnarray*}
\noindent Hence, if $n_0=0$, $n_0=n<m$ or $1\leq n_0<\text{min}\{n,m\}$ then $X(b)=0$.

\noindent Note that if $n_0=m<n$ then $\overline{f}_m=e_m$ which implies that $r(e_{m+1})=s(e_m)=r(f_m)=p$ which does not happen with our hypothesis on $a$. 

\noindent Finally, if $n_0=n=m$ then $z=x_n=a_ns_{e_n}\circ r_{e_n}^{-1}\circ E_{e_n}^r(x_{n-1})\in a_n B_{\overline{e}_n}^r$ and, since $f_m=f_n=\overline{e}_n$, we have $\rho_{f_m}(\lambda(z))\eta_{f_m}\ot b_m=\rho_{\overline{e}_n}(\lambda(x_n))\eta_{\overline{e}_n}\ot b_n\in(\rho_{\overline{e}_n}(\lambda(a_n))\eta_{\overline{e}_n}\ot 1)\cdot A_p$ and $F_p\pi_p(\lambda(z))\xi_p\cdot b_m=F_p\pi_p(\lambda(x_n))\xi_p\cdot b_m=0$. Hence,
\begin{eqnarray*}
X(b)&=&-\rho_{\overline{e}_n}(\lambda(x_n))\eta_{\overline{e}_n}\ot b_n=-\rho_{\overline{e}_n}(\lambda(a_n))\eta_{\overline{e}_n}\ot s_{e_n}\circ r_{e_n}^{-1}\circ E_{e_n}^r(x_{n-1})b_n\\
&=&-(\rho_{\overline{e}_n}(\lambda(a_n))\eta_{\overline{e}_n}\ot 1)\cdot\langle\pi_p(\lambda(a')^*)\xi_p,\pi_p(\lambda(b))\xi_p\rangle,
\end{eqnarray*}
where $a'=u_{e_n}a_{n-1}\dots u_{e_1}a_0$. It follows that, for any reduced operator $b\in P$ from $p$ to $p$ and for any $b\in A_p$ the element $X(b)$ is equal to
$$(\rho_{e_1}(\lambda(a_nu_{e_n}\dots u_{e_1}))\eta_{e_1}\ot 1)\cdot\langle\pi_p(\lambda(a_0^*))\xi_p,\pi_p(\lambda(b))\xi_p\rangle-(\rho_{\overline{e}_n}(\lambda(a_n))\eta_{\overline{e}_n}\ot 1)\cdot\langle\pi_p(\lambda(a')^*)\xi_p,\pi_p(\lambda(b))\xi_p\rangle.$$
Hence, $F_p\pi_p(\lambda(a))-\rho_p(\lambda(a))F_p$ is a finite rank operator.
\end{proof}

\noindent Since $F_p$ is a partial isometry satisfying $F_pF_p^*-1=-Q_{\xi_p}\in\mathcal{K}_{A_p}(H_p)$, It follows from Lemma \ref{LemFp} that we can apply Lemma \ref{Lem-Homotopy} to conclude that $[(K_p,\rho_p,V_p)]=0\in KK^1(P_{{\rm vert}},A_p)$, where $V_p=2F_pF_p^*-1=\bigoplus_{e\in E(\Gr),\,r(e)=p} V_e\ot 1$ and $V_e$ has been defined previously by $V_e=2Q_e-1$. It follows from the definitions that $(K_p,\rho_p,V_p)$ is a triple representing the element $\sum_{ r(e)=p}x_e^{\Gr}\underset{B_e}{\ot} [r_e]$. This concludes the proof of $(3)$.

\vspace{0.2cm}

\noindent$(4)$. Note that, for all $e\in E(\Gr)$ and all $x\in P$, we have $\EE_{\overline{e}}(\lambda(x))=\lambda(u_e)\EE_e(\lambda(u_e^*x u_e))\lambda(u_e^*)$. It follows from this formula that the operator $W_e\,:\,K_{\overline{e}}\underset{s_e^{-1}}{\ot} B_e\rightarrow K_e\underset{r_e^{-1}}{\ot} B_e$ defined by
$$W_e(\rho_{\overline{e}}(\lambda(x))\eta_{\overline{e}}\ot b)=\rho_{e}(\lambda(xu_{e}))\eta_e\ot b\quad\text{for }x\in P\text{  and }b\in B_e,$$
is a unitary operator in $\mathcal{L}_{B_e}(K_{\overline{e}}\underset{s_e^{-1}}{\ot} B_e,K_e\underset{r_e^{-1}}{\ot} B_e)$. Moreover, it is clear that $W_e$ intertwines the representations $\rho_e(\cdot)\ot 1$ and $\rho_{\overline{e}}(\cdot)\ot 1$ and we have $W_e^*(Q_e\ot 1)W_e=1\ot1-Q_{\overline{e}}\ot 1$.

\end{proof}

\begin{remark}\label{RmkFullKK1-2}
Assertions $(2)$ and $(3)$ of the preceding Proposition obviously hold for the elements $z_e^\Gr=[\lambda]\underset{P_{{\rm vert}}}{\ot} x_e^\Gr\in KK^1(P,B_e)$ and also assertions $(1)$ and $(2)$ with $\pi_{\Gr'}$ instead of $\pi_{{\rm vert}}^{\Gr'}$ since we have $\pi_{{\rm vert}}^{\Gr'}\circ\lambda_{\Gr'}=\lambda\circ\pi_{\Gr'}$ for any connected subgraph $\Gr'\subset\Gr$, with maximal subtree $\mathcal{T}'\subset\mathcal{T}$.\end{remark}

\noindent We study now in details the behavior of our elements $x_e^\Gr$ under the \textit{Serre's devissage} process.

\vspace{0.2cm}

%%%%%%%%%%%%%%%%%%%%%%%%%%%%%%%%%%%%%%%%%%%%%%%%%%
\noindent \textbf{The case of an amalgamated free product.} Let $A_1$, $A_2$ and $B$ be C*-algebras with unital faithful $*$-homomorphisms $\iota_k\,:\,B\rightarrow A_k$ and conditional expectations $E_k\,:\,A_k\rightarrow\iota_k(B)$ for $k=1,2$. Let $A_v=A_1\underset{B}{\overset{v}{*}} A_2$ be the associated vertex-reduced amalgamated free product, $A_f=A_1\underset{B}{*}A_2$ the full amalgamated free product and $\pi\,:\, A_f\rightarrow A_v$ the canonical surjection. Let $(K,\rho,\eta)$ be the GNS construction of the canonical ucp map $E\,:\,A_v\rightarrow B$ (which is the composition of the canonical surjection from $A$ to the edge-reduced amalgamated free product with the canonical ucp map from the edge-reduced amalgamated free product to $B$) and  $K_i$, for $i=1,2$, be the closed subspace of $K$ generated by $\{\rho(\pi(x))\eta\,:\,x=a_1\dots a_n\in A_f\text{ reduced and ends with }A_i\ominus B\}$. Observe that $K_i$ is a complemented Hilbert submodule of $K$. Actually we have $K=K_1\oplus K_2\oplus\eta\cdot B$. Let $Q_i\in\mathcal{L}_B(K)$ be the orthogonal projection onto $K_i$. The following Proposition is actually a special case of Lemma \ref{LemKK1}. In this special case the proof is very easy and left to the reader.

\begin{proposition}\label{Prop-KK1Free} $(K,\rho,V)$, where $V=2Q_1-1$ defines an element $x_A=[(K,\rho,V)]\in KK^1(A_v,B)$.\end{proposition}

\noindent Let $e\in E(\Gr)$ and suppose that $\Gr_e$ is not connected. We keep the same notations as the one used in the Serre's devissage process explained in the previous Section. In particular we have the $*$-isomorphism $\nu_e\,:\, A_{\Gr_e}:=P_{{\rm vert}}^{\Gr_{s(e)}}\underset{B_e}{\overset{v}{*}}P_{{\rm vert}}^{\Gr_{r(e)}}\rightarrow P_{{\rm vert}}$ from Lemma \ref{DevissageFree}. We now have two canonical elements in $KK^1(P_{{\rm vert}},B_e)$: $x_e^\Gr$ and $x_{\Gr_e}:=[\nu_e^{-1}]\underset{A_{\Gr_e}}{\ot} y_{\Gr_e}$, where $y_{\Gr_e}$ is the element associated to the vertex-reduced amalgamated free product $A_{\Gr_e}$ constructed in Proposition \ref{Prop-KK1Free}. These two elements are actually equal.

\begin{lemma}\label{LemIsoFree} We have $x_{\Gr_e}=x_e^{\Gr}\in KK^1(P_{{\rm vert}},B_e)$.\end{lemma}

\begin{proof}
The proof is a simple identification: there is not a single homotopy to write, only an isomorphism of Kasparov's triples. The key of the proof is to realize that the two ucp maps $P_{{\rm vert}}\rightarrow B_e$ defined by $\varphi=r_e^{-1}\circ\EE_e$ and $\psi=E\circ\nu_e^{-1}$ are equal, where $E\,:\, A_{\Gr_e}\rightarrow B_e$ is the canonical ucp map and it directly follows from the fact that $\nu_e$ intertwines the canonical ucp maps. Having this observation in mind, one construct an isomorphism of Kasparov's triples.

\noindent Recall that $(K_e,\rho_e,\eta_e)$ denotes the GNS construction of the ucp map $\EE_e\,:\, P_{{\rm vert}}\rightarrow B_e^r$ and $(K,\rho,\eta)$ denotes the GNS of the ucp map $E\,:\, A_{\Gr_e}\rightarrow B_e$.

\vspace{0.2cm}

\noindent Since $K=\overline{\rho\circ\nu^{-1}_{e}(P_{{\rm vert}})\eta\cdot B_e}$, $K_e\underset{r_e^{-1}}{\ot} B_e=\overline{\rho_e(P_{{\rm vert}})\eta_e\ot 1\cdot B_e}$ and
$$\langle \eta,\rho\circ\nu^{-1}_{e}(x)\eta\rangle_K=\psi(x)=\varphi(x)=\langle\eta_e\ot 1,\rho_e(x)\eta_e\ot 1\rangle_{K_e\underset{r_e^{-1}}{\ot} B_e}\quad\text{for all }x\in P_{{\rm vert}},$$
it follows that the map $U\,:\,K\rightarrow K_e\underset{r_e^{-1}}{\ot} B_e$, $U(\rho\circ\nu^{-1}_{e}(x))\eta\cdot b)=\rho_e(x)\eta_e\ot 1\cdot b$ for $x\in P_{{\rm vert}}$ and $b\in B_e$, defines a unitary $U\in\mathcal{L}_{B_e}(K,K_e\underset{r_e^{-1}}{\ot} B_e)$. Moreover, $U$ intertwines the representations $\rho\circ\nu^{-1}_{e}$ and $\rho_e(\cdot)\ot 1$. Observe that $x_{\Gr_e}$ is represented by the triple $(K,\rho\circ\nu^{-1}_{e},V)$, where $V=2Q-1$ and $Q$ is the orthogonal projection on the closed linear span of the $\rho(\pi(x_1\dots x_n))$, where $x_1\dots x_n\in P^{\Gr_{s(e)}}_{{\rm vert}}\underset{B_e}{*} P^{\Gr_{r(e)}}_{{\rm vert}}$ is a reduced operator in the free product sense and $x_n\in P^{\Gr_{s(e)}}_{{\rm vert}}$. Moreover, $x_e^{\Gr}$ is represented by the triple $(K_e\underset{r_e^{-1}}{\ot} B_e,\rho_e(\cdot)\ot 1,V_e)$, where $V_e=Q_e\ot 1$ and $Q_e$ is the orthogonal projection onto the closed linear span of the $\rho_e(\lambda(a_0u_{e_1}\dots u_{e_n}a_n))\eta_e$, where $a_0u_{e_1}\dots u_{e_n}a_n\in P$
 is reduced from $r(e)$ to $r(e)$ with $e_n=e$ and $a_n\in B_e^r$.
 
 \vspace{0.2cm}
 
 \noindent To conclude the proof, it suffices to observe that $UVU^*=V_e$.\end{proof}
\noindent We study now the case of an HNN-extension.

\vspace{0.2cm}

%%%%%%%%%%%%%%%%%%%%%%%%%%%%%%%%%%%%%%%%
\noindent \textbf{The case of an HNN extension.} For $\epsilon\in\{-1,1\}$, let $\pi_{\epsilon}\,:\,B\rightarrow A$ be a unital faithful $*$-homomorphism  $E_\epsilon\,:\,A\rightarrow B$ be a ucp map such that $E_{\epsilon}\circ\pi_{\epsilon}=\id_B$. Let $C_f$ be the full HNN-extension with stable letter $u\in\mathcal{U}(C)$, $C_v$ the vertex-reduced HNN-extension and $\pi\,:\, C_f\rightarrow C_v$ the canonical surjection. Let $(K,\rho,\eta)$ be the GNS construction of the ucp map $E=E_1\circ E_A\,:\,C_v\rightarrow B$, where $E_A\,:\,C_v\rightarrow A$ is the canonical GNS-faithful ucp map.  Define the sub $B$-module
$$K_+=\overline{\text{Span}}\{\rho(\pi(x))\eta\,:\,x=a_0u^{\epsilon_1}\dots u^{\epsilon_n}a_n\in C_f\text{ is a reduced operator with }\epsilon_n=1\text{ and }a_n\in\pi_1(B)\}.$$
Observe that $K_+$ is complemented and let $Q_+\in\mathcal{L}_B(K)$ be the orthogonal projection onto $K_+$. The following proposition, which is a special case of Lemma \ref{LemKK1}, is very easy to check.

\begin{proposition}\label{PropKK1HNN} $(K,\rho,V)$, where $V=2Q_+-1$, defines an element $x_C\in KK^1(C_v,B)$.\end{proposition}

\noindent Let $e\in E(\Gr)$ and suppose that $\Gr_e$ is connected. Up to a canonical isomorphism of $P$ we may and will assume that $\mathcal{T}\subset\Gr_e$. Recall that we have a canonical $*$-isomorphism $\nu_e\,:\,C_{\Gr_e}:={\rm HNN}_{{\rm vert}}(P_{{\rm vert}}^{\Gr_e},B_e,\pi_1,\pi_{-1})\rightarrow P_{{\rm vert}}$ defined in Lemma \ref{DevissageHNN}. As before, we get two canonical elements in $KK^1(P_{{\rm vert}},B_e)$: $x_e^\Gr$ and $x_{\Gr_e}:=[\nu_e^{-1}]\underset{C_{\Gr_e}}{\ot} y_{\Gr_e}$, where $y_{\Gr_e}\in KK^1(C_{\Gr_e},B_e)$ is the element associated to the vertex-reduced HNN-extension $C_{\Gr_e}$ constructed in Proposition \ref{PropKK1HNN}. As before, these two elements are actually equal.

\begin{lemma}\label{LemIsoHNN}We have $x_{\Gr_e}=x_e^\Gr\in KK^1(P_{{\rm vert}},B_e)$.\end{lemma}

\begin{proof}
Recall that $(K,\rho,\eta)$ denotes the GNS construction of the canonical ucp map $E\,:\, C_{\Gr_e}\rightarrow B_e$. The proof is similar to the proof of Lemma \ref{LemIsoFree} and is just a simple identification. Since $\nu_e$ intertwines the canonical ucp maps, the two ucp maps $\varphi,\psi\,:\,P_{{\rm vert}}\rightarrow B_e$ defined by $\varphi=\EE_e$ and $\psi=E\circ\nu_e^{-1}$ are equal. As before, one can deduce easily from this equality an isomorphism of Kasparov's triples. Since the arguments are the same, we leave the details to the reader.
\end{proof}
 
 \begin{remark}\label{RmkFullKK1-3}
 The analogue of Lemmas \ref{LemIsoFree}, \ref{LemIsoHNN} are obviously still valid for the elements $z_e^\Gr\in KK^1(P,B_e)$ defined in Remark \ref{RmkFullKK1-1}.
  \end{remark}
%%%%%%%%%%%%%%%%%%%%%%%%%%%%%%%%%%%%%%%%%%%%%%%
 \section{The exact sequence}
 %%%%%%%%%%%%%%%%%%%%%%%%%%%%%%%%%%%%%%%%%%%%%%
 \noindent For any separable C*-algebra $C$, let $F^*(-)$ be $KK^*(C,-)$.  It is a $\Z_2$-graded covariant functor.
 If $f$ is a morphism of C*-algebras, we will denote by $f^*$ the induced morphism.
 
 \vspace{.2cm}
 
\noindent In the sequel $P_\Gr$ or simply $P$ denotes either the full or the vertex reduced fundamental C*-algebra. We define  the boundary maps 
 $\gamma_e^\Gr$ from $F^*(P_\Gr)=KK^*(D, P_\Gr)$ to $KK^{*+1}(D,B_e)=F^{*+1}(B_e)$ by 
  $\gamma_e^\Gr(y)=y\otimes_P z^\Gr_e$ when $P$ is the full fundamental C*-algebra or $\gamma_e^\Gr(y)=y\otimes_P x^\Gr_e$ when $P$ is the vertex reduced one. In the sequel we simply write $x_e=x_e^\Gr$ and $z_e=z_e^\Gr$.

\vspace{0.2cm}

\noindent If $\Gr$ is a graph, then $E^+$ is the set of positive edges, $V$ the set of vertices and for any $v\in V$, the map from $A_v$ to $P_\Gr$ is $\pi_v$ or
sometimes $\pi_v^\Gr$ if it is necessary to indicate which graph we consider.
If one removes an edge $e_0$ (and its opposite) to $\Gr$, the new graph is called $\Gr_0$, $P_0$ is the algebra associated to it and $\pi_v^0$ is the embedding of 
$A_v$ in $P_0$. We also have for $\Gr_1\subset \Gr$ a morphism $\pi_{\Gr_1}$ from $P_{\Gr_1}$ to $P_\Gr$.

 \begin{theorem}\label{exactseqgraph} In  the presence of conditional expectations (not necessarily GNS -faithful),  we have, for $P$ the full or vertex reduced fundamental C*-algebra, a long exact sequence
 $$\longrightarrow\bigoplus_{e\in E^+} F^*(B_e)\overset{\sum_e s_e^*-r_e^*}{\longrightarrow}\bigoplus_{v\in V} F^*(A_v) \overset{\sum_v \pi_v^*}{\longrightarrow} F^*(P)\overset{\oplus_e \gamma_e^\Gr}{\longrightarrow }\bigoplus_{e\in E^+} F^{*+1}(B_e)\longrightarrow$$
 
 \end{theorem}
 
 \begin{proof}
 First note that it is indeed a chain complex.  Because $s_e$ and $r_e$ are conjugated in the full or reduced fundamental C*-algebra, we only have to check that $\gamma_e \circ \pi_v^*=0$ (which is point 2 of prop \ref{PropKK1})
 and (for $P_{vert}$),  $\sum_{e\in E^+} x_e\otimes [r_e]- x_e\otimes [s_e]=0$. As $x_e=-x_{\bar e}$ (point 4 of \ref{PropKK1}) and $s_{\bar e}=r_e$, this is the same as point 3 of \ref{PropKK1}. Because of remark 3.5, this is also true for the full fundamental C*-algebra.
 
 \vspace{0.2cm}
 
 \noindent Also if the graph contains only one geometric edge (i.e. two opposite oriented edges), we are in the case of the amalgamated free product or the HNN extension and the complex is known to be exact because of the results of \cite{FG15}. For convenience we will briefly recall why and also we will identify the boundary map. Let's do the full amalgamated free product $A_f$ first. Recall that in theorem 4.1 of \cite{FG15}, we proved that  the suspension of $A_1\underset{B}{*} A_2$ is KK-equivalent to $D$ the cone of the inclusion of $B$  in $A_1$ and $A_2$.  Obviously  $D$ fits into a short exact sequence :
 $$0\rightarrow A_1\otimes S\oplus A_2\otimes S\overset{}{\longrightarrow} D\overset{ev_0}\longrightarrow B\rightarrow 0.$$
 Therefore there is a long exact sequence for our functor $F^*$ :
 $$F^*(A_1\otimes S\oplus A_2\otimes S)\rightarrow F^*(D) \rightarrow F^*(B) \rightarrow F^{*+1}(A_1\otimes S\oplus A_2\otimes S).$$
 But $F^*(A_k\otimes S)$ identifies with $F^{*+1}(A_k)$ and $F^*(D)$ with $F^{*+1}(A_f)$. Via these identifications, the map from $F^*(B)$ to $F^*(A_k)$ becomes $i_k^*$ or its opposite (this is seen using the mapping cone exact sequence)  
 and the map from $F^*(A_k)$ to $F^*(A_f)$ is $j_k^*$. The only thing left is the identification of the boundary map from $F^*(A_f)$ to $F^{*+1}(B)$. It is obviously the Kasparov product by $x\otimes [ev_0]$ where $x$ is the
 element in $KK^1(A_f,D)$ that implements the K-equivalence. The element $x\ot[ev_0]\in KK^1(A_f,B)$ has been described in lemma 4.9 of \cite{FG15} and it is equal to $[\pi]\ot x_A$, where $x_A\in KK^1(A_v,B)$ is exactly the element of \ref{Prop-KK1Free} and $\pi$ is the canonical surjection from the full amalgamated free product $A_f$ to the vertex-reduced amalgamated free product $A_v$. Therefore the boundary map
 is exactly given by the corresponding $\gamma_e^{\Gr}$ for the graph of the free product. Moreover, since $x$ actually factorizes as $[\pi]\otimes_{A_v} z$ where $z\in KK^1(A_v,D)$, the same identifications and the same exact sequence hold for the vertex reduced free product $A_v$ and theorem \ref{exactseqgraph} is true for free products.

\vspace{0.2cm}
  
\noindent Now let's tackle the HNN extension case. Let's call $C_m$ the full HNN extension of $(A,B,\theta)$ and $E$ and $E_\theta$ the conditional expectations from $A$ to $B$ and $\theta(B)$. We also denote by $C_v$ the vertex-reduced HNN-extension and $\pi\,:\, C_m\rightarrow C_v$ the canonical surjection. An explicit  isomorphism is known to exist between $C_m$ and the full amalgamated free product $e_{11}M_2(A)\underset{B\oplus B}{*} M_2(B)e_{11}$ where $B\oplus B$ imbeds diagonally  in $M_2(A)$ via the canonical inclusion and $ \theta$, $e_{11}$ is the matrix unit $\begin{pmatrix}1&0\\0&0\\  \end{pmatrix}$ and the conditional expectations are 
 $E_1\begin{pmatrix} a_1&a_2\\ a_3&a_4\\  \end{pmatrix}=E(a_1)\oplus E_\theta(a_4)$ from $M_2(A)$ to $B\oplus B$ and $E_2\begin{pmatrix} b_1&b_2\\ b_3&b_4\\  \end{pmatrix}=b_1\oplus b_4$
 from $M_2(B)$ to $B\oplus B$. The exact sequence for the HNN extension is then deduced from this isomorphism of C*-algebras (cf. \cite{Ue08} for example).

\vspace{0.2cm}

\noindent If we call $j_A$ and $j_B$ the inclusions of $M_2(A)$ respectively $M_2(B)$ in the free product  then the unitary $u$ in $C_m$ that implements $\theta$ is mapped to 
 $j_A\begin{pmatrix}0&1\\0&0\\  \end{pmatrix}j_B\begin{pmatrix}0&0\\1&0\\  \end{pmatrix}$.

\noindent It is then clear that a reduced word in $C_m$ that ends with $u$ times $b$ with $b$ in $B$ is mapped into a reduced word in the free product that ends with 
 $j_B\begin{pmatrix}0&0\\1&0\\  \end{pmatrix}\begin{pmatrix}b&0\\0&0\\  \end{pmatrix}=j_B\begin{pmatrix} 0&b'\\ b&0\\ \end{pmatrix}e_{11}$ i.e. that ends in $j_B(M_2(B))\ominus (B\oplus  B)$.
 Therefore, in this situation and after a Kasparov product by $[\pi]$ on the left, the element described in \ref{PropKK1HNN} is the same as the element described in \ref{Prop-KK1Free} and we have identified the correct boundary map.
 
 \vspace{0.2cm}
 
\noindent Let's have a look now at the vertex reduced situation. Observe that the conditional expectation $E_2$ from $M_2(B)$ to $B\oplus B$ is GNS faithful. It follows from the constructions of \cite[Section 2]{FG15} that $M_2(A)\overset{2}{\underset{B\oplus B}{*}} M_2(B)$ is isomorphic to $M_2(A)\overset{e}{\underset{B\oplus B}{*}} M_2(B)$ and as a consequence $M_2(A)\overset{1}{\underset{B\oplus B}{*}} M_2(B)$ is isomorphic to
$M_2(A)\overset{v}{\underset{B\oplus B}{*}} M_2(B)$. Using the universal properties it is now obvious that the vertex reduced HNN extension of $(A,B,\theta)$ is $e_{11}M_2(A)\overset{1}{\underset{B\oplus B}{*}} M_2(B)e_{11}$. Therefore the identification described earlier for the full free product and HNN extension is again true for the vertex reduced free product and corresponding vertex reduced HNN extension. Hence theorem \ref{exactseqgraph} is again valid for HNN extensions.

\vspace{0.2cm}
 
\noindent We now prove exactness at each place by induction on the cardinal of edges and \textit{devissage}. Note that \ref{LemIsoFree} and \ref{LemIsoHNN}  allow us to decompose our fundamental algebra in
 HNN or free product while using the same boundary maps $\gamma_e$.
 
 \begin{lemma}
 We have the exactness of $\bigoplus_{e\in E^+} F^*(B_e)\overset{\sum_e s_e^*-r_e^*}{\longrightarrow}\bigoplus_{v\in V} F^*(A_v) \overset{\sum_v \pi_v^*}{\longrightarrow} F^* ( P )$.
 \end{lemma}
 
 \begin{proof}
 Choose a positive edge $e_0$.  Then without this edge (and its opposite), the graph $\Gr_0$ is either connected (Case I) or has two connected components $\Gr_1$ and $\Gr_2$ (Case II).
 
 \vspace{0.2cm}
 
 \noindent\textbf{Case I.} $P$ is the HNN extension of $P_{\Gr_0}$ and $B_{e_0}$. The set of vertices of $\Gr$ is the same as the set of vertices of $\Gr_0$ and we may and will assume that $v_0=s(e_0)=r(e_0)$. Let $x=\oplus x_v$ be in $\oplus_{v\in V} F^*(A_v)$ such that $\sum_v \pi_v^*(x_v)=0$. If $y=\sum_v {\pi_v^0}^*(x_v)$, then clearly $\pi_{\Gr_0}(y)=0$. Then, the long exact sequence for 
 $P$ seen as  an HNN extension implies then that there exists $y_0\in F^*(B_{e_0})$ such that $(\pi_{v_0}\circ s_{e_0})^*(y_0)-(\pi_{v_0}\circ r_{e_0})^*(y_0)=y=\sum_v {\pi_v^0}^*(x_v)$. Hence,
 $$\sum_v {\pi_v^0}^*(\oplus_{v\neq v_0} x_v\oplus (x_{v_0}-s_{e_0}^*(y_0)+r_{e_0}^*(y_0))=0.$$
 Using the exactness for $P_0$ as $\Gr_0$ has one less edge, we get that there exists for any $e\neq e_0$ a $y_e$ such that $\sum_{e\neq e_0} s_e^*(y_e)-r_e^*(y_e)= \oplus_{v\neq v_0} x_v\oplus (x_{v_0}-s_{e_0}^*(y_0)+r_{e_0}^*(y_0))$. Thus,
 $$\sum_{e\neq e_0} s_e^*(y_e)-r_e^*(y_e)+s_{e_0}^*(y_0)-r_{e_0}^*(y_0)=x.$$
 
 \noindent\textbf{Case II.} $P$ is the amalgamated free product of $P_1=P_{\Gr_1}$ and $P_2=P_{\Gr_2}$ over $B_{e_0}$. For $i=1,2$, denote by $V_i$ the vertices of $\Gr_i$. We know that $V$ is the disjoint union of $V_1$ and $V_2$. The map $\pi_v^i$ will be the embedding of $A_v$ in $P_i$. We also write $v_1=s(e_0)$ and $v_2=r(e_0)$. Let $x=\oplus x_v$ be in $\oplus_{v\in V} F^*(A_v)$ such that $\sum_v \pi_v^*(x_v)=0$. Let  $x_i=\oplus_{v\in V_i} {\pi_v^i}^*(x_v)$. Clearly $\pi_{\Gr_1}^*(x_1)+\pi_{\Gr_2}^*(x_2)=0$. Then, the long exact sequence for $P$ seen as an amalgamated free product gives a $y_0\in F^*(B_{e_0})$ such that 
 $(\pi_{v_1}^1\circ s_{e_0})^*(y_0)-(\pi_{v_2}^2\circ r_{e_0})^*(y_0)=x_1\oplus x_2$. Define $\bar x_1= \oplus_{v\in V_1} x_v -s_{e_0}^*(y_0)$ and $\bar x_2= \oplus_{v\in V_2} x_v +r_{e_0}^*(y_0)$.
We have, for $i=1,2$, $\sum_{v\in V_i}{ \pi_v^i}^*(\bar x_i)=0$. Therefore by induction as $\Gr_i$ has strictly less edges than $\Gr$, there exists for any $e\neq e_0$ a $y_e\in F^*(B_e)$ such that
 $\bar x_1\oplus \bar x_2= \sum_{e\neq e_0} s_e^*(y_e)-r_e^*(y_e)$. Hence, $x= \sum_{e\neq e_0} s_e^*(y_e)-r_e^*(y_e)+ s_{v_0}^*(y_0)- r_{v_0}^*(y_0)$.\end{proof}

 \begin{lemma}
The following chain complex is exact in the middle
$$\bigoplus_{v\in V} F^*(A_v) \overset{\sum_v \pi_v^*}{\longrightarrow} F^*(P)\overset{\oplus_e \gamma_e^\Gr}{\longrightarrow }\bigoplus_{e\in E^+} F^{*+1}(B_e)$$
 \end{lemma}
 
 \begin{proof} As in the previous Lemma, we separate in the proof in Case I and Case II.
 
 \vspace{0.2cm}
 \noindent\textbf{Case I.} Let $x$ be in $F^*(P)$ such that for any $e$, $\gamma_e^\Gr(x)=0$. In particular for the edge $e_0$.
 Using the long exact sequence for $P$ seen as an HHN extension, and since $\gamma_{e_0}^\Gr(x)=0$ we get that there exists $x_0$ in $F^*(P_0)$ such that $\pi_{ \Gr_0}^*(x_0)=x$. 
 For any edges $e\neq e_0$, one has $\gamma_e^{\Gr_0}(x_0)=\gamma_e^{\Gr}( \pi_{ \Gr_0}^*(x_0))=0$. Hence by induction there exists for any $v\in V(\Gr_0)=V(\Gr)$ a $y_v\in F^*(A_v)$
 such that $\sum_v {\pi_v^0}^*(y_v)=x_0$. Hence $x=\sum_v (\pi_{ \Gr_0}\circ\pi_v^0)^*(y_v)=\sum_v \pi_v^*(y_v)$.
 
 \vspace{0.2cm}
 \noindent\textbf{Case II.} Using that $P$ is the free product of $P_1$ and $P_2$, we get an $x_i\in F^*(P_i)$ for $i=1,2$ such that $x=\pi_{\Gr_1}^*(x_1)+\pi_{\Gr_2}^*(x_2)$. 
 Now for $i=1,2$, and for any edge $e$ of $\Gr_i$, we have
 $$\gamma_e^{\Gr_i}(x_i)=\gamma_e^{\Gr}( \pi_{ \Gr_i}^*(x_i))=\gamma_e^\Gr(x)-\gamma_e^\Gr(\pi_{\Gr_j}^*(x_j))\text{ for }j\neq i.$$
 But $e$ is not an edge of $\Gr_j$, so $\gamma_e^\Gr\circ\pi_{\Gr_j}^*=0$. Hence $\gamma_e^{\Gr_i}(x_i)=0$.
 By induction we get for any vertex of $V_1\cup V_2=V(\Gr)$ a $y_v\in F^*(A_v)$ such that $x_i=\sum_{v\in V_i} {\pi_v^i}^*(y_v)$ for $i=1,2$. 
 Therefore $x=\sum_v \pi_v^*(y_v)$.
 
 \end{proof}
 
 \begin{lemma}
 The following chain complex is exact in the middle
 $$F^{*-1}(P)\overset{\oplus_e \gamma_e^\Gr}{\longrightarrow }\bigoplus_{e\in E^+} F^{*}(B_e)\overset{\sum_e s_e^*-r_e^*}{\longrightarrow}\bigoplus_{v\in V} F^*(A_v) $$
 
 \end{lemma}
 
 \begin{proof}
\textbf{Case I.} Let $x=\oplus_{e\in E^+} x_e$ such that $\sum_e s_e^*(x_e)-r_e^*(x_e)=0$. Then for the distinguished vertex $v_0$, one has
 ${\pi_{v_0}^0}^*(s_{e_0}^*(x_{e_0}))- {\pi_{v_0}^0}^*(r_{e_0}^*(x_{e_0}))= -\sum_{e\neq e_0} {\pi_{v_0}^0}^*(s_{e}^*(x_{e_0}))- {\pi_{v_0}^0}^*(r_{e}^*(x_{e_0}))$. But as $e$ is an edge of $\Gr_0$, $s_e$ and $r_e$ are conjugated by a unitary of $P_0$. Therefore there difference are $0$ in any KK-groups. 
 Thus ${\pi_{v_0}^0}^*(s_{e_0}^*(x_{e_0}))- {\pi_{v_0}^0}^*(r_{e_0}^*(x_{e_0}))=0$. Using the long exact sequence for $P$ as an HHN extension, we get 
 a $y_0$ in $F^{*-1}(P)$ such that $\gamma_{e_0}^\Gr(y_0)=x_{e_0}$.  Set now $\bar x_e= x_e-\gamma_e^\Gr(y_0)$ for any $e\neq e_0$ and compute
 \begin{eqnarray*}
 \sum_{e\neq e_0} s_e^*(\bar x_e)-r_e^*(\bar x_e)&=&\sum_{e\neq e_0} s_e^*( x_e)-r_e^*( x_e)-\sum_e s_e^*(\gamma_e^\Gr(y_0))- r_e^*(\gamma_e^\Gr(y_0))\\&&+ s_{e_0}^*(\gamma_{e_0}^\Gr(y_0))- r_{e_0}^*(\gamma_{e_0}^\Gr(y_0)) \\
 &=&\sum_{e} s_e^*( x_e)-r_e^*( x_e),\\
 \end{eqnarray*}
 by the third property of $\gamma_e$. Hence, $\sum_{e\neq e_0} s_e^*(\bar x_e)-r_e^*(\bar x_e)=0$. By induction there exists $\bar y_1$ in $F^{*-1}(P_0)$ such that for all $e\neq e_0$, $\gamma_e^{\Gr_0}(y_1)=\bar x_e$. Set at last $ y_1=\pi_{\Gr_0}^*(\bar y_1)$ which is an element of $F^{*-1}(P)$.
 Now $\gamma_{e_0}^\Gr(y_0+y_1)= x_0+\gamma_{e_0}^\Gr\circ \pi_{\Gr_0}^* (\bar y_1)$. But $e_0$ is not an edge of $\Gr_0$ so $\gamma_{e_0}^\Gr\circ \pi_{\Gr_0}^*=0$.
 Hence $\gamma_{e_0}^\Gr(y_0+y_1)= x_0$. On the other end, for $e\neq e_0$, $\gamma_{e}^\Gr(y_0+y_1)= \gamma_{e}^\Gr(y_0)+\bar x_e$ as $\gamma_e^{\Gr_0}=\gamma_e^\Gr\circ \pi_{\Gr_0}^*$.
 It follows that $\gamma_{e}^\Gr(y_0+y_1)=x_e$.
 
\vspace{0.2cm}
\noindent\textbf{Case II.} Call $E_i$ the  edges of $\Gr_i$ for $i=1,2$. Note that for any positive edge $e$, if $s(e)\in V_1$ then either $e\in E_1$ or $e=e_0$ and
 if $r(e)\in V_2$ then $e\in E_2$. Let $x=\oplus_{e\in E^+} x_e$ such that $\sum_e s_e^*(x_e)-r_e^*(x_e)=0$. The equality can be rewritten as
 $\sum_{e\in E_1^+} s_e^*(x_e)-r_e^*(x_e)+ s_{e_0}^*(x_{e_0})=0$ in $\oplus_{v\in V_1} F^*(A_v)$ and $\sum_{e\in E_2^+} s_e^*(x_e)-r_e^*(x_e)- r_{e_0}^*(x_{e_0})=0$ in
 $\oplus_{v\in V_2} F^*(A_v)$. Let's compute now $\pi_{v_1}^1(x_{e_0})$.  It is $- \sum_{e\in E_1^+} (\pi_{s(e)}^1\circ s_e)^*(x_e)-(\pi_r(e)^1\circ r_e)^*(x_e)$ by the preceding remark. But as $s_e$ and
 $r_e$ are conjugated in $P_1$ because $e$ is an edge of $\Gr_1$, this is $0$.
 In the same way $\pi_{v_2}^2(x_{e_0})=0$.
 Therefore using the long exact sequence for $P$ as a free product of $P_1$ and $P_2$, there is a $y_0$ in $F^{*-1}(P)$ such that $\gamma_{e_0}^\Gr(y_0)=x_{e_0}$. For all $e\neq e_0$ set $\bar x_e=x_e-\gamma_e^\Gr(y_0)$.
 Then,
 $$\sum_{e\in E_1^+} s_e^*(\bar x_e)-r_e^*(\bar x_e)=\sum_{e\in E_1^+} s_e^*(x_e)-r_e^*(x_e) -\left (\sum_{e\in E_1^+} s_e^*\circ \gamma_e^\Gr(y_0)-r_e^*\circ \gamma_e^\Gr(y_0)\right).$$
 But the third property of the $\gamma_e^\Gr$ implies that $0=\sum_{e\in E_1^+} s_e^*\circ \gamma_e^\Gr +s_{e_0}^*\circ \gamma_{e_0}^\Gr-\sum_{e\in E_1^+} r_e^*\circ \gamma_e^\Gr$
 using the remark made at the begining of this proof. Hence,
 $$\sum_{e\in E_1^+} s_e^*(\bar x_e)-r_e^*(\bar x_e)=\sum_{e\in E_1^+} s_e^*(x_e)-r_e^*(x_e)+ s_{e_0}^*(x_{e_0})=0.$$
 Similarly, $\sum_{e\in E_2^+} s_e^*(\bar x_e)-r_e^*(\bar x_e)=0$. Therefore by induction, there exists for $i=1,2$, an element $y_i$ in $F^{*-1}(P_i)$ such that 
 for all $e$ in $E_i^+$, $\gamma_e^{\Gr_i}(y_i)=\bar x_e$.  Set now $y=y_0+\pi_{\Gr_1}(y_1)+\pi_{\Gr_2}(y_2)$ in $F^{*-1}(P)$.
 Then $\gamma_{e_0}^\Gr(y)=x_{e_0}+\gamma_{e_0}^{\Gr}\circ\pi_{\Gr_1}^*(y_1)+\gamma_{e_0}^{\Gr}\circ\pi_{\Gr_2}^*(y_2)= x_{e_0}$ as
 $\gamma_{e_0}^{\Gr}\circ\pi_{\Gr_i}=0$ since $e_0$ is not an edge of $\Gr_1$ nor $\Gr_2$.
 On the other end, for $e\in E_1$, $\gamma_{e}^\Gr(y)=\gamma_{e}^\Gr(y_0)+\gamma_{e}^{\Gr_1}(y_1)+0$ as $e$ is not an edge of $\Gr_2$.
 Hence $\gamma_{e}^\Gr(y)=\gamma_{e}^\Gr(y_0)+\bar x_e=x_e$.  The same is of course true for an edge in $E_2$. So we are done.
 \end{proof}
 \noindent The proof of Theorem \ref{exactseqgraph} is now complete.\end{proof}
 
\noindent Let's treat now the case $F^*(-)=KK(-,C)$. Again if $f$ is a morphism of C*-algebras we will adopt the same notation $f^*$ for the induced morphism. 
Now the map $\gamma_e^\Gr$ from $F(B_e)$ to $F(P)$ is defined as $\gamma_e^\Gr(a)=x_e^\Gr\otimes_{B_e} a$ if $P$ is the vertex reduced fundamental C*-algebra or
$\gamma_e^\Gr(a)=z_e^\Gr\otimes_{B_e} a$ if $P$ is the full fundamental C*-algebra.

\begin{theorem}\label{exactseqgraph2} In the presence of  conditional expectations, we have, for $P$ the full or reduced fundamental C*-algebra, a long exact sequence
 $$\longleftarrow\bigoplus_{e\in E^+} F^*(B_e)\overset{\sum_e s_e^*-r_e^*}{\longleftarrow}\bigoplus_{v\in V} F^*(A_v) \overset{\sum_v \pi_v^*}{\longleftarrow} F^*(P)\overset{\oplus_e \gamma_e^\Gr}{\longleftarrow }\bigoplus_{e\in E^+} F^{*+1}(B_e)\longleftarrow$$\end{theorem}

\begin{proof}
As before this is a chain complex and the same identifications proves it for free products and HNN extension. We will now show exactness with the three following lemmas.

 \begin{lemma}
 We have the exactness of  $\bigoplus_{e\in E^+} F^*(B_e)\overset{\sum_e s_e^*-r_e^*}{\longleftarrow}\bigoplus_{v\in V} F^*(A_v) \overset{\sum_v \pi_v^*}{\longleftarrow} F^*( P )$.
 \end{lemma}
 
\begin{proof}
Let $x=\oplus x_v\in \oplus_v F(A_v)$ such that $\sum_e s_e^*(\oplus x_v)-r_e^*(\oplus x_v)=0$.

\vspace{0.2cm}

\noindent\textbf{Case I.} We have $\sum_{e\neq e_0}  s_e^*(\oplus x_v)-r_e^*(\oplus x_v)=0$ hence, there is a $y_0$ in $F(P_0)$ such that for all $v$, ${\pi_v^0}^*(y_0)=x_v$. But $s_{e_0}^*\circ{\pi_{v_0}^0}^*(y_0)=s_{e_0}^*(x_{v_0})=r_{e_0}^*(x_{v_0})=r_{e_0}^*\circ{\pi_{v_0}^0}^*(y_0)$. Using the exact sequence for $P$ as an HNN extension of $P_0$ and the two copies of $B_{e_0}$, we get that there is $y\in F(P)$ such that ${\pi_{\Gr_0}}^*(y)=y_0$.
 Now for all $v$, $\pi_v^*(y)={\pi_v^0}^*(y_0)=x_v$.
 
 \vspace{0.2cm}

\noindent\textbf{Case II.} We have, for $k=1,2$, $\sum_{e\in E_k^+} s_e^*(\oplus x_v)-r_e^*(\oplus x_v)=0$ hence there is $y_k\in F(P_k)$ such that ${\pi_v^k}^*(y_k)=x_v$ for any $v\in V_k$. As $s_{e_0}^*\circ {\pi_{v_1}^1}^*(y_1)= s_{e_0}^*(x_{v_1})=r_{e_0}^*(x_{v_2})=r_{e_0}^*\circ {\pi_{v_2}^2}^*(y_2)$. Using the exact sequence for $P$ as a free product, we have a $y\in F(P)$ such that ${\pi_{\Gr_k}}^*(y)=y_k$ for $k=1,2$. Then for $k=1,2$ and all $v\in V_k$,  $\pi_v^*(y)={\pi_v^k}^*(y_k)=x_v$.
\end{proof}
 
\begin{lemma}
The following chain complex is exact in the middle
$$\bigoplus_{v\in V} F^*(A_v) \overset{\sum_v \pi_v^*}{\longleftarrow} F^*(P)\overset{\oplus_e \gamma_e^\Gr}{\longleftarrow }\bigoplus_{e\in E^+} F^{*+1}(B_e)$$
\end{lemma}
 
\begin{proof} Let $y$ be in $F(P)$ such that $\pi_v^*(y)=0$ for all $v$. 

\vspace{0.2cm}

\noindent\textbf{Case I.} Let $y_0=\pi_{\Gr_0}^*(y).$ Then for all $v$, ${\pi_v^0}^*(y_0)= \pi_v^*(y)=0$. Therefor there exists $x=\sum_{e\neq e_0} x_e$ such that 
$\sum_{e\neq e_0} {\gamma_e^{\Gr_0}}(x_e)=y_0$.  Put $z=y-\sum_{e\neq e_0} {\gamma_e^{\Gr}}^*(x_e)$. Then,
$$\pi_{\Gr_0}^*(z)=y_0-\sum_{e\neq e_0} {\gamma_e^{\Gr_0}}(x_e)=0.$$
Hence there is a $x_{e_0}\in F(B_{e_0})$ such that $\gamma_{e_0}(x_{e_0})=z$ and $y= \sum_{e\neq e_0} {\gamma_e^{\Gr}}(x_e)+\gamma_{e_0}(x_{e_0})$.

\vspace{0.2cm}

\noindent\textbf{Case II.}  Let $y_k=\pi_{\Gr_k}^*(y)$ for $k=1,2$.  For all $v\in V_k$, ${\pi_v^k}^*(y_k)= \pi_v^*(y)=0$, hence there exists $x_k=\oplus_{e\in E_k^+} x_e$ such that
$\sum_{e\in E_k^+} \gamma_e^{\Gr_k}(x_e)=y_k$. Let $z=y-\sum_{e\neq e_0} {\gamma_e^{\Gr}}^*(x_e)$. Then for $k=1,2$, $\pi_{\Gr_k}^*(z)= y_k-\sum_{e\in E_k^+} \gamma_e^{\Gr_k}(x_e)=0$ as $\pi_{\Gr_2}^*\circ\gamma_e^{\Gr_1}=0$ because of \ref{PropKK1}.  Hence $z=\gamma_{e_0}(x_{e_0})$ for some $x_{e_0}$ in $F(B_{e_0})$ and
we are done.

\end{proof}

\begin{lemma}
The following chain complex is exact in the middle
$$F^{*-1}(P)\overset{\oplus_e \gamma_e^\Gr}{\longleftarrow }\bigoplus_{e\in E^+} F^{*}(B_e)\overset{\sum_e s_e^*-r_e^*}{\longleftarrow}\bigoplus_{v\in V} F^*(A_v) $$
\end{lemma}
 
 \begin{proof}Let $x=\oplus x_e$ in $F(\oplus_e B_e)$ such that $\sum_{e\in E^+}\gamma_e^\Gr(x_e)=0$.
 
\vspace{0.2cm}

\noindent\textbf{Case I.} We have $0={\pi_{\Gr_0}}^*(\sum_{e\in E^+}\gamma_e^\Gr(x_e))=\sum_{e\neq e_0} \gamma_e^{\Gr_0}(x_e)$ as ${\pi_{\Gr_0}}^*\circ\gamma_{e_0}=0$.
 Hence by induction, there is a $z=\oplus z_v$ in $\oplus_v F(A_v)$ such that for all $e\neq e_0$, $x_e=s_e^*(z_{s(e)})-r_e^*(z_{r(e)})$. Put $x_0=x_{e_0}-s_{e_0}^*(z_{v_0})-r_{e_0}^*(z_{v_0})$. By 5.5 we have $\sum_{e\in E^+} \gamma_e\circ s_e^*-\gamma_e\circ r_e^*=0$ hence,
 $$\gamma_{e_0}\circ (-s_{e_0}^*(z_{v_0})+r_{e_0}^*(z_{v_0}))=\sum_{e\neq e_0} \gamma_e( s_e^*(\oplus z_v))-\gamma_e(r_e^*(\oplus z_v))=\sum_{e\neq e_0} \gamma_e(x_e).$$
 It follows that $\gamma_{e_0}(x_0)=\gamma_{e_0}(x_{e_0})+\sum_{e\neq e_0} \gamma_e(x_e)=0$.  Using the long exact sequence for $P$ as an HNN extension of $P_0$ and $B_{e_0}$, we get a $z_0\in F(P_0)$ such that $x_0= s_{e_0}^*({\pi_{v_0}^0}^*(z_0))-r_{e_0}^*({\pi_{v_0}^0}^*(z_0))$.  
 So $x_{e_0}=s_{e_0}^*(z_{v_0}+{\pi_{v_0}^0}^*(z_0))- r_{e_0}^*(z_{v_0}+{\pi_{v_0}^0}^*(z_0))$ and we are done.
 
\vspace{0.2cm}

\noindent\textbf{Case I.} $0={\pi_{\Gr_k}}^*(\sum_{e\in E^+}\gamma_e^\Gr(x_e))=\sum_{e\neq E_k^+} \gamma_e^{\Gr_k}(x_e)$ for $k=1,2$. Hence there is a $z=\oplus z_v$
 such that for all $e\in E_k^+$,   $x_e=s_e^*(z_{s(e)})-r_e^*(z_{r(e)})$. Write $x_0=x_{e_0}-s_{e_0}^*(z_{v_1})-r_{e_0}^*(z_{v_2})$. As before we have
 that $\gamma_{e_0}(x_0)=0$ and by exactness of the exact sequence for the free product of $P_1$ and $P_2$ there is $z_1\in F(P_1)$ and $z_2\in F(P_2)$ such that 
 $x_0=s_{e_0}^*({\pi_{v_1}^1}^*(z_1))-r_{e_0}^*({\pi_{v_2}^2}^*(z_2))$. Finally $x_{e_0}=s_{e_0}^*(z_{v_1}+{\pi_{v_1}^1}^*(z_1))- r_{e_0}^*(z_{v_2}+{\pi_{v_2}^2}^*(z_2))$.
\end{proof}

\noindent The proof of Theorem \ref{exactseqgraph2} is now complete.\end{proof}

%%%%%%%%%%%%%%%%%%%%%%%%%%%%%%%%%%%%%%%%%%%%
\section{Applications}
 %%%%%%%%%%%%%%%%%%%%%%%%%%%%%%%%%%%%%%%%%%%%

\noindent In this section we collect some applications of our results to $K$-equivalence and $K$-amenability of quantum groups.
\vspace{0.2cm}

\noindent Let $(\Gr,A_p,B_e)$ and $(\Gr,A_p',B_e')$ be two graphs of unital C*-algebras with maps $s_e$ and $s_e'$ and conditional expectations $E_e^s$ and $(E_e^s)' $. Suppose that we have unital $*$-homomorphisms $\nu_p\,:\,A_p\rightarrow A_p'$ and $\nu_e\,:\,B_e\rightarrow B_e'$ such that $\nu_e=\nu_{\overline{e}}$ and $\nu_{s(e)}\circ s_e=s_e'\circ\nu_e$ for all $e\in E(\Gr)$. Let  $P$ and $P$ be the associated full fundamental C*-algebras with canonical unitaries $u_e$ and $u_e'$ respectively. By the relations $\nu_e=\nu_{\overline{e}}$ and $\nu_{s(e)}\circ s_e=s_e'\circ\nu_e$ and the universal property of the full fundamental C*-algebra, there exists a unique unital $*$-homomorphism $\nu\,:\, P\rightarrow P'$ such that
$$\nu\vert_{A_p}=\nu_p\quad\text{and}\quad\nu(u_e)=u_e'\quad\text{for all }p\in V(\Gr),\,e\in E(\Gr).$$

\begin{theorem}\label{ThmKequivalence}
If $(E_e^s)'\circ\nu_{s(e)}=\nu_{s(e)}\circ E_e^s$ and $\nu_p,\nu_e$ are $K$-equivalences for all $p\in V(\Gr),e\in E(\Gr)$ then $\nu$ is a $K$-equivalence.
\end{theorem}

\begin{proof}
Consider the following diagrams with exact rows
{\Small
$$\begin{array}{ccccccccc}
\underset{{e\in E^+}}{\bigoplus}KK(D,B_e)&\rightarrow&\underset{{p\in V}}{\bigoplus}KK(D,A_p)&\rightarrow&KK(D,P)&\rightarrow &\underset{{e\in E^+}}{\bigoplus}KK^1(D,B_e)&\rightarrow &\underset{{p\in V}}{\bigoplus}KK^1(D,A_p)\\
\downarrow\bigoplus\cdot\underset{B_e}{\ot}[\nu_e]& &\downarrow \bigoplus\cdot\underset{A_p}{\ot}[\nu_p]& &\downarrow \cdot\underset{P}{\ot}[\nu]& &\bigoplus\cdot\underset{B_e}{\ot}[\nu_e]& &\downarrow \bigoplus\cdot\underset{A_p}{\ot}[\nu_p]\\\underset{{e\in E^+}}{\bigoplus}KK(D,B'_e)&\rightarrow&\underset{{p\in V}}{\bigoplus}KK(D,A'_p)&\rightarrow&KK(D,P')&\rightarrow &\underset{{e\in E^+}}{\bigoplus}KK^1(D,B'_e)&\rightarrow &\underset{{p\in V}}{\bigoplus}KK^1(D,A'_p)\\
\end{array}$$
}

{\Small
$$\begin{array}{ccccccccc}
\underset{{e\in E^+}}{\bigoplus}KK(B'_e,D)&\rightarrow&\underset{{p\in V}}{\bigoplus}KK(A'_p,D)&\rightarrow&KK(P',D)&\rightarrow &\underset{{e\in E^+}}{\bigoplus}KK^1(B'_e,D)&\rightarrow &\underset{{p\in V}}{\bigoplus}KK^1(A'_p,D)\\
\downarrow\bigoplus [\nu_e] \underset{B_e}{\ot}\cdot& &\downarrow \bigoplus [\nu_p]\underset{A_p}{\ot}\cdot& &\downarrow [\nu]\underset{P}{\ot}\cdot& &\bigoplus [\nu_e]\underset{B_e}{\ot}\cdot& &\downarrow \bigoplus [\nu_p]\underset{A_p}{\ot}\cdot\\
\underset{{e\in E^+}}{\bigoplus}KK(B_e,D)&\leftarrow&\underset{{p\in V}}{\bigoplus}KK(A_p,D)&\leftarrow&KK(P,D)&\leftarrow &\underset{{e\in E^+}}{\bigoplus}KK^1(B_e,D)&\leftarrow &\underset{{p\in V}}{\bigoplus}KK^1(A_p,D)\\
\end{array}$$
}

\noindent By the Five Lemma and the hypothesis, it suffices to check that, for each $D$, every square of the two diagrams is commutative. We check that for the first diagram. The verification for the second diagram is similar. For a unital inclusion $X\subset Y$ of unital C*-algebras, we write $\iota_{X\subset Y}$ the inclusion map. The first square on the left and the last square on the right of the first diagram are obviously commutative since, by hypothesis, $\nu_{s(e)}\circ s_e=s_e'\circ\nu_e$ and $\nu_{r(e)}\circ r_e=r_e'\circ\nu_e$ for all $e\in E^+$. The second square on the left is commutative since, by definition of $\nu$, we have $\nu\circ\iota_{A_p\subset P}=\iota_{A_p'\subset P'}\circ\nu_p$ for all $p\in V$. Hence, it suffices to check that the third square, starting from the left, is commutative. Note that the commutativity of this square is equivalent to the equality $z_e\underset{B_e}{\ot}[\nu_e]=[\nu]\underset{P'}{\ot}z_e'\in KK^1(P,B'_e)$, where $z_e\in KK^1(P,B_e)$ and $z_e'\in KK^1(P',B_e')$ are the $KK^1$ elements constructed in Remark 3.7 associated with the graphs of C*-algebras $(\Gr,A_p,B_e)$ and $(\Gr,A'_p,B'_e)$ respectively. This equality follows easily from the assumption $(E_e^s)'\circ\nu_{s(e)}=\nu_{s(e)}\circ E_e^s$ since it gives a canonical isomorphism of Hilbert modules $K_e\otimes_{\nu_e} B_e'\simeq K_e'$ which is easily seen to implement an isomorphisms between the Kasparov triples representing $z_e\underset{B_e}{\ot}[\nu_e]$ and $[\nu]\underset{P'}{\ot}z_e'$.
\end{proof}

\noindent We write $P_{\rm vert}$ the vertex reduced fundamental C*-algebra of $(\Gr,A_p,B_e)$ and $\lambda\,:\, P\rightarrow P_{\rm vert}$ the canonical surjective unital $*$-homomorphism. The following Theorem is an immediate consequence of the two $6$ terms exact sequences we proved in this paper: one for the full fundamental C*-algebra $P$ and one for the vertex-reduced fundamental C*-algebra $P_{\rm{vert}}$ and the Five Lemma.

\begin{theorem}\label{CorKequivalenceGraphs}
Suppose that $\Gr$ is a finite graph then the class of the canonical surjection $[\lambda]\in KK(P,P_{{\rm vert}})$ is invertible.
\end{theorem}

\begin{remark}
The previous result is actually true without assuming the graph $\Gr$ to be  finite. Indeed the inverse of $[\lambda]$ and the homotopy showing that it is an inverse can be constructed directly, without using induction. Since such a proof requires more work and does not bring any new ideas, we chose to not include it.
\end{remark}

\begin{corollary}
The following holds.
\begin{enumerate}
\item If $G$ be the fundamental compact quantum group of a finite graph of compact quantum groups $(G_p,G_e,\Gr)$ then $\widehat{G}$ is $K$-amenable if and only if $\widehat{G}_p$ is $K$-amenable for all $p$.
\item If $G$ is the compact quantum group obtained from the (finite) graph product of the family of compact quantum groups $G_p$, $p\in V(\Gr)$ (see \cite{CF14}) then $\widehat{G}$ is $K$-amenable if and only if $\widehat{G}_p$ is $K$-amenable for all $p\in V(\Gr)$.
\end{enumerate}
\end{corollary}

\begin{proof}
Using induction, $(2)$ is a consequence of $(1)$ since, as observed in \cite{CF14}, a graph product maybe written as an amalgamated free product using a kind of \textit{devissage} strategy.

\vspace{0.2cm}

\noindent Let us prove $(1)$. Consider the two graphs of C*-algebras $(\Gr,C_{max}(G_p),C_{max}(G_e))$ and\\ $(\Gr,C_{red}(G_p),C_{red}(G_e))$ with full fundamental C*-algebra $P_{max}$ and $P$ respectively. Note that both graphs have natural families of conditional expectations by only the conditional expectations on $(\Gr,C_{red}(G_p),C_{red}(G_e))$  are GNS faithful (expect in the presence of co-amenability) L.et $P_{red}$ be the vertex reduced fundamental C*-algebra of $(\Gr,C_{red}(G_p),C_{red}(G_e))$. We recall that $C_{max}(G)=P_{max}$ and $C_{red}(G)=P_{red}$ (see \cite{FF13}). Let $\lambda\,:\,P\rightarrow P_{red}$ the canonical surjection, which is a $K$-equivalence by Theorem \ref{CorKequivalenceGraphs} and let $\nu\,:\,P_{max}\rightarrow P$ be the canonical surjection obtained from the canonical surjections $\nu_p:=\lambda_{G_p}\,:\,C_{max}(G_p)\rightarrow C_{red}(G_p)$ and $\nu_e:=\lambda_{G_e}\,:\,C_{max}(G_e)\rightarrow C_{red}(G_e)$ as explained in the discussion before Thereom \ref{ThmKequivalence}. Since the hypothesis on the conditional expectations of this Theorem are obviously satisfied, it follows that, whenever $\widehat{G}_p$ is $K$-amenable for all $p$ (hence $\widehat{G}_e$ is also $K$-amenable for all $e$ as a quantum subgroup of $\widehat{G}_{s(e)}$), $\widehat{G}$ is $K$-amenable. The proof of the converse is obvious.
\end{proof}

\begin{remark}
The first assertion of the previous Corollary strengthens the results of \cite[Corollary 19]{Pi86} and also \cite{FF13, Fi13, Ve04} and unify all the proofs.
\end{remark}

\bigskip

\noindent
{\sc Emmanuel Germain} \\
  {LMNO, CNRS UMR 6139, Universit\'e de Caen, France\\
\em E-mail address: \tt emmanuel.germain@unicaen.fr}

\bigskip

\noindent
{\sc Pierre FIMA} \\ \nopagebreak
  {Univ Paris Diderot, Sorbonne Paris Cit\'e, IMJ-PRG, UMR 7586, F-75013, Paris, France \\
  Sorbonne Universit\'es, UPMC Paris 06, UMR 7586, IMJ-PRG, F-75005, Paris, France \\
  CNRS, UMR 7586, IMJ-PRG, F-75005, Paris, France \\
\em E-mail address: \tt pierre.fima@imj-prg.fr}

\end{document}